\documentclass[12pt, leqno]{amsart}

\usepackage[OT2,T1]{fontenc}
\usepackage[utf8]{inputenc}
\usepackage{amstext}
\usepackage{amsthm}
\usepackage{amsopn}
\usepackage{amsfonts}
\usepackage{amsmath}
\usepackage{latexsym}
\usepackage[francais,english]{babel}
\usepackage{amscd}
\usepackage{amssymb}
\usepackage{amsmath}
\usepackage[all,cmtip]{xy}
\usepackage{leftidx}
\usepackage{graphicx}
\usepackage{etoolbox}
\patchcmd{\section}{\normalfont\scshape\centering}{\normalfont\bfseries}{}{}
\patchcmd{\subsection}{-.5em}{.5em}{}{}

\newtheorem{theo}{{Theorem}}[section]
\newtheorem{coro}[theo]{{Corollary}}
\newtheorem{lemma}[theo]{{Lemma}}
\newtheorem{prop}[theo]{Proposition}

\theoremstyle{definition}
\newtheorem{remark}[theo]{\textbf{Remark}}
\newtheorem{defn}[theo]{Definition}
\newtheorem{example}[theo]{Example}
\numberwithin{equation}{section}

\newcommand{\ra}{\rightarrow}

\newcommand{\ol}{\overline}

\newcommand{\cI}{\mathcal{I}}

\newcommand{\cS}{\mathcal{S}}

\newcommand{\cR}{\mathcal{R}}

\newcommand{\rR}{\mathrm{R}}
\newcommand{\rI}{\mathrm{I}}

\newcommand{\id}{\mathrm{id}}

\newcommand{\Gal}{\mathrm{Gal}}

\newcommand{\gm}{\mathbb{G}}

\newcommand{\rat}{\mathbb{Q}}
\newcommand{\ent}{\mathbb{Z}}

\DeclareFontEncoding{OT2}{}{} 
  \newcommand{\textcyr}[1]{%
    {\fontencoding{OT2}\fontfamily{wncyr}\fontseries{m}\fontshape{n}%
     \selectfont #1}}
\newcommand{\sha}{{\mbox{\textcyr{Sh}}}}

\begin{document}
\tolerance 400 \pretolerance 200 \selectlanguage{english}

\title{The Tate-Shafarevich groups of multinorm-one tori}
\author{T.-Y. Lee}
\date{\today}
\maketitle

\begin{abstract}
Let $k$ be a global field and $L$ be a product of cyclic extensions of $k$.
Let $T$ be the torus defined by the multinorm equation $N_{L/k}(x)=1$ and let $\hat{T}$ be its character group.
The Tate-Shafarevich group and the algebraic Tate-Shafarevich group of $\hat{T}$ in degree 2 give
obstructions to the Hasse principle and weak approximation for rational points on principal homogeneous spaces of $T$.
We give concrete descriptions of these groups and provide several examples.
\end{abstract}

\section{Introduction}
Let $k$ be a global field and  fix a separable closure $k_s$ of $k$.
In the following all the separable extensions of $k$ are considered as subfields of $k_s$.

Let $K_i$ be a finite separable extension of $k$ for $i=0,...,m$.
Set $L=K_0\times...\times K_m$.
Let $T_{L/k}$ be the torus defined by the multinorm equation:
\begin{equation}\label{e}
N_{L/k}(t)=1.
\end{equation}
Denote by $\hat{T}_{L/k}$ the character group of ${T}_{L/k}$.

Let $\Omega_k$ be the set of all places of $k$.
Define
\begin{equation}
\sha^i(k,T_{L/k}):=\ker (H^i(k,T_{L/k})\to \underset{v\in\Omega_k}{\prod}H^i(k_v,T_{L/k})).
\end{equation}

It is well-known that the elements in $\sha^1(k,T_{L/k})$ are in one-to-one correspondence with the isomorphism classes of $T_{L/k}$-torsors which have $k_v$-points for all $v\in\Omega_k$.
To be precise, let $X_c$ be the variety defined by
\begin{equation}\label{e0}
N_{L/k}(t)=c,
\end{equation}
where $c\in k^{\times}$.
Suppose that $X_c$ has a $k_v$-point for all $v\in\Omega_k$. Then
$X_c$ corresponds to an element $[X_c]\in\sha^1(k,T_{L/k})$.
By Poitou-Tate duality, the class $[X_c]$ defines a map $\sha^2(k,\hat{T}_{L/k})\to\rat/\ent$,
which is the Brauer-Manin obstrution to the Hasse principle for the existence of rational points of $X_c$.
Hence the group $\sha^2(k,\hat{T}_{L/k})$ is related to the local-global principle for multinorm equations.

For a  Galois module $M$ over $k$, define
$$\sha^i_{\omega}(k,M):=\{[C]\in H^i(k,M)\ \mbox{such that } [C]_v=0\ \mbox{for almost all $v\in\Omega_k$}.\}$$
It is clear that $\sha^i(k,\hat{T}_{L/k})\subseteq\sha^i_{\omega}(k,\hat{T}_{L/k})$.
The case $i=2$ is the most interesting to us.
In fact
if $\sha^2_{\omega}(k,\hat{T}_{L/k})=\sha^2(k,\hat{T}_{L/k})$,
weak approximation holds for $T_{L/k}$ and hence for those $X_c$ with a $k$-point (\cite{San} Prop. 8.9 and Thm. 8.12).

The local-global principle and weak approximation for multinorm equations (\ref{e0}) have been extensively studied.
One can see \cite{PoR}, \cite{Po}, \cite{DW}, \cite{BLP18} and \cite{M} for  recent developments on this topic.
In this paper, we are interested in the groups $\sha^2(k,\hat{T}_{L/k})$ and $\sha^2_\omega(k,\hat{T}_{L/k})$
(and hence the group $\sha^2_\omega(k,\hat{T}_{L/k})/\sha^2(k,\hat{T}_{L/k})$).
These groups measure the  obstruction to the local-global principle for existence of rational points of $X_c$ and the obstruction to weak approximation.

Under the assumption that $L$ is a product of (not necessarily disjoint) cyclic extensions of \emph{prime-power degrees},
We give a formula for $\sha^2_\omega(k,\hat{T}_{L/k})$ and $\sha^2(k,\hat{T}_{L/k})$.
Briefly speaking, the group $\sha^2_\omega(k,\hat{T}_{L/k})$ is determined by the "maximal bicyclic field" $M$ generated by subfields of $K_i$ and $\sha^2(k,\hat{T}_{L/k})$ is determined by the "maximal bicyclic and locally cyclic subfield" of $M$. In combination with \cite{BLP18} Proposition 8.6, one can calculate the group
$\sha^2(k,\hat{T}_{L/k})$ for $L$ a product of cyclic extensions of arbitrary degrees.
This generalizes the result in \cite{BLP18} \S 8.
Furthermore we compute the bigger group $\sha^2_\omega(k,\hat{T}_{L/k})$ which is related to weak approximation.
We give several concrete examples in the final section.

The paper is structured as follows.
Section 1 introduces the notation.
In Section 2 we give a combinatorial description of $\sha^2(k,\hat{T}_{L/k})$ and $\sha^2_\omega(k,\hat{T}_{L/k})$.
In Section 3, we prove some preliminaries about cyclic extensions, which will be the main tools in the following sections. In Section 4-6, we define the \emph{patching degree} and the \emph{degree of freedom} in order to describe the generators of the group $\sha^2(k,\hat{T}_{L/k})$ (resp. $\sha^2_\omega(k,\hat{T}_{L/k})$).
We give  formulas for $\sha^2(k,\hat{T}_{L/k})$ and $\sha^2_\omega(k,\hat{T}_{L/k})$ in Section 7 and provide several examples in the last section.

\section{Notation and definitions}

For a $k$-algebra A and a place $v\in\Omega_k$, we denote $A\otimes_k k_v$ by $A^v$.

A finite Galois extension $F$ of $k$ is said to be \emph{locally cyclic} at $v$ if $F\otimes_k k_v$ is a product of cyclic extensions of $k_v$. $F$ is said to be locally cyclic if it is locally cyclic at all $v\in\Omega_k$.

A bicyclic extension $F/k$ is a Galois extension with $\Gal(F/k)$ isomorphic to $\ent/n_1\ent\times \ent/n_2\ent$ where $n_1$, $n_2>1$ and $n_2|n_1$.

Throughout this paper, we assume $\underset{i=0}{\overset{m}{\cap}} K_i=k$.

\section{Preliminaries on algebraic tori}

For a $k$-torus $T$, we denote by $\hat{T}$ its character group  as a $\Gal(k_s/k)$-module.

Let $A$ be a field and $A'$ be a finite dimensional $A$-algebra.
For an $A'$-torus $T$, we denote by $\rR_{A'/A}(T)$  its Weil restriction to $A$.
(For more details on Weil restriction, see \cite{CGR} A.5.)

Let $N_{A'/A}$ be the norm map and denote by $T_{A'/A}$ the norm one torus $R^{(1)}_{A'/A}(\gm_m)$.

We first prove some general facts about multinorm-one tori defined by finite separable extensions of $k$.

We recall the following well-known fact (\cite{San} Lemma 1.9).
\begin{lemma}\label{weak approximation for permutation}
Let $\mathcal{G}=\Gal(k_s/k)$ and $M$ be a permutation module of $\mathcal{G}$. Then $\sha_\omega^2(k,M)=0$.
\end{lemma}

\medskip

\medskip

Recall some notation defined in \cite{BLP18}.
Denote the index set by $\cI=\{1,...,m\}$ and $\cI'=\{0\}\cup\cI$.
In the following, we always assume that $m\geq 2$.

Set \begin{itemize}
 \item[$\bullet$] $K'=\underset{i\in\cI}{\prod} K_i$,
 \item[$\bullet$] $L=\underset{i\in\cI'}{\prod}K_i$,
 \item[$\bullet$]  $E = K_0 \otimes_k K'$, and
 \item[$\bullet$] $E_i = K_0 \otimes_k K_i$.
\end{itemize}

\medskip
The norm maps $N_{K_0/k} : K_0 \to k$ and $N_{K'/k} : K' \to k$ induce
$N_{E/K'} : E \to K'$ and $N_{E/K_0} : E \to K_0$.
Let
$\phi:\rR_{E/k}(\gm_m)\ra \rR_{L/k}(\gm_m)$ be defined by $\phi(x) = (N_{E/K_0}(x)^{-1},N_{E/K'}(x))$.
It is clear that
the image of $\phi$ is contained in $T_{L/k}$. Moreover, $\phi$ is surjective onto $T_{L/k}$ as a map of algebraic groups (easily checked after base change
to the separable closure $k_s$ of $k$).

\medskip
Consider the torus $S_{K_0,K'}$ defined by the exact sequence
\begin{equation}\label{e2}
\xymatrix@C=0.5cm{
  1 \ar[r] & S_{K_0,K'} \ar[r] & \rR_{E/k}(\gm_m) \xrightarrow{\phi} T_{L/k} \ar[r] & 1 }.
\end{equation}
Note that $S_{K_0,K'}$ also fits in the exact sequence

\begin{equation}\label{e3}
\xymatrix@C=0.5cm{
  1 \ar[r] &  S_{K_0,K'} \ar[r] & \underset{i\in I}{\prod}\rR_{K_i/k}(T_{E_i/K_i}) \xrightarrow{N_{E/K_0}} T_{K_0/k}  \ar[r] & 1 }.
\end{equation}


\begin{prop}\label{cyclicnorm}
Let $K_0$ be a cyclic extension of arbitrary degree. Then
$\sha_\omega^2(k, \hat{T}_{K_0/k})=0$.

\end{prop}

\begin{proof}

Let $\sigma$ be a generator of ${\rm Gal}(K_0/k)$.
Consider the exact sequence
$$1 \ra \gm_m \ra \rR_{K_0/k}(\gm_m) \ra  T_{K_0/k}  \ra 1,$$
where the map from  $\rR_{K_0/k}(\gm_m)$ to $ T_{K_0/k}$ sends $x$ to $x/\sigma(x)$.
Its dual sequence  is
$$1 \ra \hat{T}_{K_0/K} \ra \rI_{K_0/k}(\ent) \ra  \ent  \ra 1.$$

By Lemma \ref{weak approximation for permutation} we have $\sha^2_\omega(k,\rI_{K_0/k}(\ent))=0$.
As $H^1(k,\ent)$=0,
we have $\sha^2_\omega(k,\hat{T}_{K_0/k}) = 0$.
\end{proof}

\begin{lemma}\label{T and S}
We have
\begin{enumerate}
  \item $\sha^{2}(k,\hat{T}_{L/k})\simeq\sha^{1}(k,\hat{S}_{K_0,K'})$.
  \item $\sha^{2}_\omega(k,\hat{T}_{L/k})\simeq\sha^{1}_\omega(k,\hat{S}_{K_0,K'})$.
\end{enumerate}

\end{lemma}
\begin{proof}
The first statement is \cite{BLP18} Lemma 3.1.

We now prove (2).
Consider the dual sequence of (\ref{e2}):
\begin{equation}\label{e4}
\xymatrix@C=0.5cm{
  0 \ar[r] & \hat{T}_{L/k} \ar[r] & \rI_{E/k}(\ent) \xrightarrow{\phi} \hat{S}_{K_0,K'} \ar[r] & 0 }.
\end{equation}
The exact sequence \ref{e4} gives rise to the following exact sequence:
\begin{equation}\label{e5}
\xymatrix@C=0.5cm{
0 \ar[r] & H^1(k, \hat{S}_{K_0,K'}) \xrightarrow{\delta}  H^2(k,\hat{T}_{L/k})\xrightarrow{} H^2(k,\rI_{E/k}(\ent))}.
\end{equation}
By Lemma \ref{weak approximation for permutation}, we have $\sha_\omega^2(k,\rI_{E/k}(\ent))=0$.
Therefore $\sha_\omega^2(k,\hat{T}_{L/k})$ is in the image of $\delta$.
Let $[\theta]$ be an element in $H^1(k,\hat{S}_{K_0,K'})$ such that $\delta[\theta]\in\sha_\omega^2(k,\hat{T}_{L/k})$.
As $H^1(k_v,\rI_{E/k}(\ent))=0$ for all $v\in\Omega_k$ , the element $[\theta]_v=0$ if $(\delta[\theta])_v=0$.
Hence $[\theta]\in\sha_\omega^1(k,\hat{S}_{K_0,K'})$.
The lemma then follows.
\end{proof}

\subsection{Combinatorial description of Tate-Shafarevich groups}
From now on we assume that $K_0$ is a cyclic extension of degree $p^{\epsilon_0}$ and
we denote by $K_0({f})$ the unique subfield of $K_0$ of degree $p^f$.

For all $i\in \cI$, we set
\begin{itemize}
  \item[$\bullet$] $p^{e_{0,i}}=[K_0\cap K_i:k]$, and
  \item[$\bullet$] $e_i=\epsilon_0-e_{0,i}$.
\end{itemize}

As $K_0$ is cyclic, for each $i\in I$, the algebra $K_0\otimes_k K_i$ is a product of cyclic extensions of degree $p^{e_i}$ of $K_i$.
Without loss of generality, we assume that $e_i\geq e_{i+1}$.
Since we assume that $K_0\cap(\underset{i}{\cap} K_i)=k$, we have $e_{0,1}=0$ and $e_1=\epsilon_0$.

We can assume further that for any $i\neq j$, $K_j\nsubseteq K_i$.
To see this, suppose that there are distinct $i$, $j$ such that $K_j\subseteq K_i$.
Set $J=\{0,1,...,m\}\setminus \{i\}$ and set $L'=\underset{i\in J}{\prod} K_i$.
Then $T_{L/k}\simeq T_{L'/k}\times R_{K_i/k}(\gm_m)$.
By Lemma \ref{weak approximation for permutation}, $\sha^2(k,\hat{T}_{L/k})\simeq\sha^2(k,\hat{T}_{L'/k})$ and $\sha^2_\omega(k,\hat{T}_{L/k})\simeq\sha^2_\omega(k,\hat{T}_{L'/k})$.

Recall some definitions from \cite{BLP18}.
Let $s$ and $t$ be  positive integers. For $s\geq t$, let $\pi_{s,t}$ be the canonical projection  $\ent/p^{s}\ent \to
\ent/p^{t}\ent$.
For $x\in\ent/p^{s}\ent$ and $y\in\ent/p^{t}\ent$, we say that \emph{$x$ dominates $y$} if $s\geq t$ and $\pi_{s,t}(x)=y$; if this is the case, we write
 $x\succeq y$.
For $x\in\ent/p^{s}\ent$ and $y\in\ent/p^{t}\ent$, let  $\delta(x,y)$  be the greatest nonnegative integer $d\leq \min\{s,t\}$
such that $\pi_{s,d}(x)=\pi_{t,d}(y)$.
We have  $\delta(x,y)=\min\{s,t\}$ if and only if $x\succeq y$ or $y\succeq x$.

Recall that $e_i\geq e_{i+1}$ for $i=1,...,m-1$.
For $a=(a_1,...,a_m)\in\underset{i \in \cI}{\oplus}\ent/p^{e_i}\ent$ and $n\in \ent/p^{e_1}\ent$,
let $I_n(a)$ be the set
$\{i\in\cI|\ n\succeq a_i\}$ and let $I(a)=(I_0(a),...,I_{p^{e_1}-1}(a))$.

\medskip




Given a positive integer $0\leq d\leq \epsilon_0$ and $i \in \cI$, let $\Sigma_i^d$ be the set of all places $v\in\Omega_k$ such that at each place $w$ of $K_i$ above $v$, the following equivalent conditions hold
(See \cite{BLP18} Prop. 5.5 and 5.6.):
\begin{itemize}
  \item[(1)] The algebra $K_0 \otimes_k K_i^w$ is isomorphic to a product of
  isomorphic field extensions of degree at most $p^d$ of $K^w_i$.
  \item[(2)] $K_0(\epsilon_0-d)\otimes_k K_i^w$ is isomorphic to
   a product of $K^w_i$.
\end{itemize}

\medskip

Let $\Sigma_i=\Sigma^0_i$. In other words,
$\Sigma_i$ is the set of all places $v \in \Omega_k$ where $K_0 \otimes K_i^v$ is isomorphic to a product of copies of $K_i^v$.

\medskip
Let $a=(a_1,...,a_m)$ be an element in $\underset{i \in \cI}{\oplus}\ent/p^{e_i}\ent$ and  $I(a)=(I_0,...,I_{p^{e_1}-1})$.
For $I_n\subsetneqq \cI$, define
\begin{equation}\label{e24}
\Omega(I_n)=\underset{i\notin I_n}{\cap}{\Sigma}_i^{\delta(n,a_i)}.
\end{equation}
For $I_n=\cI$, we set $\Omega(I_n)=\Omega_k$.

\medskip
Set \[G = G(K_0,K') =\{(a_1,...,a_m)\in\underset{i \in \cI}{\oplus}\ent/p^{e_i}\ent |\ \underset{n\in\ent/p^{e_1}\ent}{\bigcup} \Omega (I_n(a))=\Omega_k\},
\]
and set $D$ to be the diagonal subgroup generated by $(1,1,...,1)$.

Define $\sha(K_0,K')$ as $G(K_0,K')/D$.
\begin{theo}\label{isom between G and sha}(\cite{BLP18} Cor. 5.4) The Tate-Shafarevich group
$\sha^2(k,\hat{T}_{L/k})$ is isomorphic to $\sha(K_0,K')$.
\end{theo}
\begin{proof}
This follows from Lemma \ref{T and S} and \cite{BLP18} Thm.  5.3.
\end{proof}

\medskip


Next we give a combinatorial description of $\sha_\omega^2(k,\hat{T}_{L/k})$, which is similar to the description of $\sha^2(k,\hat{T}_{L/k})$.

For $a=(a_1,...,a_m)\in\underset{i \in \cI}{\oplus}\ent/p^{e_i}\ent$, we define \[\mathcal{S}_a=
\Omega_k\backslash(\underset{n\in\ent/p^{e_1}\ent}{\bigcup} \Omega (I_n(a))).\]
Set
\[G_\omega = G_{\omega}(K_0,K') =\{(a_1,...,a_m)\in\underset{i \in \cI}{\oplus}\ent/p^{e_i}\ent |\ \cS_a\mbox{ is a finite set}\}.
\]
Clearly $G\subseteq G_{\omega}$.
Define $\sha_{\omega}(K_0,K')$ as $G_\omega(K_0,K')/D$, where $D$ is the subgroup generated by the diagonal element $(1,...,1)$.
We prove an analogue of Theorem \ref{isom between G and sha}.

\begin{theo}\label{sha_omega T0_primepower}
{\it Keep the notation  above. Then $\sha^2_\omega(k,\hat{T}_{L/k})\simeq \sha_\omega(K_0,K')$.}
\end{theo}
\begin{proof}
By Lemma \ref{T and S}, it is sufficient to show that $\sha^1_\omega(k,\hat S_{K_0,K'})\simeq \sha_\omega(K_0,K')$.
The proof is similar to the proof of \cite{BLP18} Theorem 5.3.
We sketch the proof here. For more details one can refer to \cite{BLP18}.

Consider the dual sequence  of (\ref{e3}),
\begin{equation}\label{e30}
\xymatrix{
  0 \ar[r] & \hat{T}_{K_0/k} \xrightarrow{\iota}  \rI_{K'/k}(\hat{T}_{E/K'}) \xrightarrow{\rho} \hat{S}_{K_0,K'} \ar[r] & 0
 },
\end{equation}
and the exact sequence induced by (\ref{e30}),
\begin{equation}\label{e20}
\xymatrix{
   H^1(k,  \hat{T}_{K_0/k}) \xrightarrow{\iota^1}   H^1(k, \rI_{K'/k}(\hat{T}_{E/K'})) \xrightarrow{\rho^1} H^1(k,  \hat{S}_{K_0,K'} ) \ra
   H^2(k, \hat{T}_{K_0/k})
 . }
\end{equation}


By \cite{BLP18} Lemma 1.2 and Lemma 1.3, we can identify $H^1(k,\hat T_{K_0/k})$ to  $\ent/p^{\epsilon_0}\ent$ and
$H^1(k,\rI_{K_i/k}(\hat{T}_{E_i/K_i}))$ to $\ent/p^{e_i}\ent$ for $1\leq i\leq m$.
Under this identification,  we can rewrite the exact sequence $(\ref{e20})$ as follows :
\begin{equation}\label{e22}
\xymatrix{
   \ent/p^{\epsilon_0}\ent \xrightarrow{\iota^1} \underset{i \in \cI}{{\oplus}} \ent/p^{e_i}\ent
   \xrightarrow{\rho^1} H^1(k,\hat{S}_{K_0,K'} ) \ra
   H^2(k, \hat{T}_{K_0/k}),
 }
\end{equation}
where $\iota^1$ is the natural projection from $\ent/p^{\epsilon_0}\ent$ to $\ent/p^{e_i}\ent$ for each $i$. Note that the image of $\iota^1$ is the subgroup $D$, and we
have the exact sequence

\begin{equation}\label{e21}
\xymatrix{
   0 \ra ( \underset{i \in \cI}{{\oplus}} \ent/p^{e_i}\ent )/D  \xrightarrow{\rho^1} H^1(k,\hat{S}_{K_0,K'} ) \ra
   H^2(k, \hat{T}_{K_0/k}).
 }
\end{equation}

By Proposition \ref{cyclicnorm} the group $ \sha_\omega^1(k,\hat{S}_{K_0,K'} )$ is contained in the image of $\rho^1$.
Let $a=(a_1,...,a_m)\in \underset{i \in \cI}{{\oplus}} \ent/p^{e_i}\ent$ and $[a]$ be its image in $(\underset{i \in \cI}{{\oplus}} \ent/p^{e_i}\ent)/D$. We claim that $\rho^1([a])$ is in $ \sha_\omega^1(k,\hat{S}_{K_0,K'} )$ if and only if $a\in G_\omega$.

For $v\in\Omega_k$, we denote by $a^v$ the image of $a$ in
$\underset{i=1}{\overset{m}{\oplus}}H^1(k_v,\rI_{K^v_i/k_v}(\hat{T}_{E^v_i/K^v_i}))$, and by $D_v$ the image of $D$ in this sum.

By the exact sequence  (\ref{e21}) over $k_v$, we have $\rho^1([a])\in  \sha^1_\omega(k,\hat{S}_{K_0,K'} )$ if and only if  $a^v \in D_v$ for almost all places $v \in \Omega_k$.

Note that $a^v=(n,...,n)^v$ if and only  if $v\in \Omega(I_n(a))$. Hence $a^v\in D_v$ if and only if $v \in \underset{n\in\ent/p^{e_1}\ent}{\bigcup} \Omega (I_n(a))$.

Our claim then follows.

\end{proof}

\subsection{Subtori}
For $0\leq r\leq \epsilon_0$, we set the following:
\begin{itemize}
\item[$\bullet$] $U_r=\{i\in\cI|\ e_{0,i}=r\}$.
\item[$\bullet$] $K_{U_r}=\underset{i\in U_r}{\prod} K_i$.
\item[$\bullet$] $L_r=K_0\times K_{U_r}$.
\item[$\bullet$] $E_{U_r}=K_0\otimes_k K_{U_r}$.
\end{itemize}

\medskip

Pick an $r$ such that $U_r$ is nonempty.
We define $S_{K_0,K_{U_r}}$ as in (\ref{e2}) and $(\ref{e3})$.
Namely let
$\phi_r:\rR_{E_{U_r}/k}(\gm_m)\ra \rR_{L_r/k}(\gm_m)$ be defined by $\phi_r(x) = (N_{E_{U_r}/K_0}(x)^{-1},N_{E_{U_r}/K_{U_r}}(x))$ and define $S_{K_0,K_{U_r}}$ by the following exact sequence.
\begin{equation}\label{e31}
\xymatrix@C=0.5cm{
  1 \ar[r] & S_{K_0,K_{U_r}} \ar[r] & \rR_{E_{U_r}/k}(\gm_m) \xrightarrow{\phi} T_{L_r/k} \ar[r] & 1 },
\end{equation}
The torus $S_{K_0,K_{U_r}}$ also fits in the exact sequence:
\begin{equation}\label{e32}
\xymatrix@C=0.5cm{
  1 \ar[r] &  S_{K_0,K_{U_r}} \ar[r] & \underset{i\in U_r}{\prod}\rR_{K_{i}/k}(T_{E_{i}/K_{i}}) \xrightarrow{N_{E_{U_r}/K_0}} T_{K_0/k}  \ar[r] & 1 }.
\end{equation}

Write $\rR_{K'/k}(\gm_m)$ as $\underset{i\in U_r}{\prod}\rR_{K_i/k}(\gm_m)\times\underset{i\in \cI\setminus U_r}{\prod}\rR_{K_i/k}(\gm_m)$.
There is a natural injective group homomorphism
$$\alpha_r:\underset{i\in U_r}{\prod}\rR_{K_i/k}(\gm_m)\to\underset{i\in U_r}{\prod}\rR_{K_i/k}(\gm_m)\times\underset{i\in \cI\setminus U_r}{\prod}\rR_{K_i/k}(\gm_m),$$
which sends $x$ to $(x,1)$.
Then $\alpha_r$ induces an injective homomorphism $\alpha_{E_{U_r}}$ from
$\rR_{E_{U_r}/k}(\gm_m)\to\rR_{E/k}(\gm_m)$, and an injective homomorphism $\id_{K_0}\times\alpha_r$ from $\rR_{K_0/k}(\gm_m)\times\underset{i\in U_r}{\prod}\rR_{K_i/k}(\gm_m)\to \rR_{K_0/k}(\gm_m)\times\rR_{K'/k}(\gm_m)$.
It is easy to check the following diagram commutes.
\begin{equation}\label{e33}
\CD
  1 @>>> S_{K_0,K'} @>>> \rR_{E/k}(\gm_m) @>\phi>>T_{L/k}@>>>1 \\
   @AAA @A\alpha_{U_r} AA @A\alpha_{U_r} AA  @A \id_{K_0}\times\alpha_r AA @AAA\\
  1  @>>> S_{K_0,K_{U_r}} @>>> \rR_{E_{U_r}/k}(\gm_m) @>\phi_r>>T_{L_r/k}@>>>1
\endCD
\end{equation}
Together with Lemma \ref{T and S}, we have
\begin{equation}\label{e34}
\CD
    \sha^1_\omega(k,\hat{S}_{K_0,K'}) @>\sim>> \sha^2_\omega(k,\hat{T}_{L/k}) \\
   @V\hat{\alpha}_{U_r} VV   @V \id_{K_0}\times\hat{\alpha}_r VV \\
   \sha^1_\omega(k,\hat{S}_{K_0,K_{U_r}}) @>\sim>> \sha^2_\omega(k,\hat{T}_{L_r/k}).
\endCD
\end{equation}

Note that for $i\in U_r$, we have $e_i=\epsilon_0-r$.
We define $G(K_0,K_{U_r})$ and $G_\omega(K_0,K_{U_r})$ by replacing $\cI$ with $U_r$ as in Section 2.1 and 2.2.
Namely for $a=(a_i)_{i\in U_r}\in\underset{i \in U_r}{\oplus}\ent/p^{e_i}\ent$, we define $\mathcal{S}_a$ to be the set
$\Omega_k\backslash(\underset{n\in\ent/p^{\epsilon_0-r}\ent}{\bigcup} \Omega (I_n(a)))$.
Set
\[G_{\omega}(K_0,K_{U_r}) =\{(a_i)_{i\in U_r}\in\underset{i \in U_r}{\oplus}\ent/p^{e_i}\ent |\ \cS_a\mbox{ is a finite set}\},
\]
and set $G(K_0,K_{U_r})$ to be the subset of $G_{\omega}(K_0,K_{U_r})$ consisting of all elements $a$ with $\cS_a=\emptyset$.
Consider the natural projection  $\varpi_r:\underset{i\in I}{\oplus}\ent/p^{e_i}\ent\to \underset{i\in U_r}{\oplus}\ent/p^{e_i}\ent$.
Then the natural projection induces a homomorphism $G_\omega(K_0,K')/D\to G_{\omega}(K_0,K_{U_r})/D,$
which we still denote by $\varpi_r$.
Note that by Theorem \ref{sha_omega T0_primepower}, we have isomorphisms $\sha^2_\omega(k,\hat{T}_{L/k})\simeq G_\omega(K_0,K')/D$ and $\sha^2_\omega(k,\hat{T}_{L_r/k})\simeq G_\omega(K_0,K_{U_r})/D$.
\begin{prop}
The morphism $\id_{K_0}\times\hat{\alpha}_r:\sha^2_\omega(k,\hat{T}_{L/k})\to \sha^2_\omega(k,\hat{T}_{L_r/k})$ coincides with $\varpi_r$.
\end{prop}
\begin{proof}
It is enough to show that the map induced by ${\alpha}_{E_{U_r}}$ from $\sha^1_\omega(k,\hat{S}_{K_0,K'})\to \sha^1_\omega(k,\hat{S}_{K_0,K_{U_r}})$ is equal to $\varpi_r$.

The map $\alpha_{E_{U_r}}$ gives a map between character groups $$\hat{\alpha}_{E_{U_r}}:\rI_{E/k}(\ent)=\rI_{E_{U_r}/k}(\ent)\underset{i\in\cI\setminus U_r}{\oplus}\rI_{E_i/k}(\ent)\to\rI_{E_{U_r}/k}(\ent),$$
which is the natural projection.
Hence the map from $\rI_{K'/k}(\hat{T}_{E/K'})=\rI_{K_{U_r}/k}(\hat{T}_{E_{U_r}/K_{U_r}})\underset{i\in\cI\setminus U_r}{\oplus}\rI_{K_i/k}(\hat{T}_{E_i/K_i})$ to $\rI_{K_{U_r}/k}(\hat{T}_{E_{U_r}/K_{U_r}})$ induced by ${\alpha}_{E_{U_r}}$ (restricted to $\rR_{K_{U_r}/k}({T}_{E_{U_r}/K_{U_r}})$) is the natural projection.

Therefore the map induced by
$\hat{\alpha}_{E_{U_r}}$ from $H^1(k,\rI_{K'/k}(\hat{T}_{E/K'}))\simeq \underset{i\in\cI}{\oplus}\ent/p^{e_i}\ent$ to
$H^1(k,\rI_{K_{U_r}/k}(\hat{T}_{E_{U_r}/K_{U_r}}))\simeq\underset{i\in U_r}{\oplus}\ent/p^{e_i}\ent$  is the natural projection.

By exact sequences (\ref{e22}) and (\ref{e21}), we have $\hat{\alpha}_{E_{U_r}}:\sha^1_\omega(k,\hat{S}_{K_0,K'})\to \sha^1_\omega(k,\hat{S}_{K_0,K_{U_r}})$ is equal to $\varpi_r$.
\end{proof}




\section{Preliminaries on cyclic extensions}

From now on we assume that $K_i$ are \emph{cyclic extensions} of $k$.

Let $p$ be a prime which divides $[L:k]$, and let $L(p)$ be the largest subalgebra of $L$ such that $[L(p):k]$ is a power of $p$.
By \cite{BLP18} Proposition 8.6, to compute $\sha^2(k,\hat{T}_{L/k})$ it is enough to compute $\sha^2(k,\hat{T}_{L(p)/k})$ for each such $p$.
Hence in the following we assume that $[L:k]$ is a \emph{power of $p$} unless we state otherwise.

By renaming these cyclic extensions, we always assume that the degree of $K_0$ is minimal.
Let  $p^{\epsilon_i}=[K_i:k]$ for all $i\in \cI'$.
For a nonnegative integer $f\leq \epsilon_i$, we denote by $K_i({f})$ the unique subfield of $K_i$ of degree $p^f$.

For all $i\in \cI$, we set $p^{e_{i,j}}=[K_i\cap K_j:k]$.
As we assume that $K_j\nsubseteq K_i$ for any $i$, $j\in \cI'$,
 $e_{i,j}<\min \{\epsilon_i,\epsilon_j\}$ for all $i,j\in\cI'$.

Note that for $i,j\in \cI$ with $i<j$, we have $e_{i,j}\geq e_{0,i}$.
This follows from  the assumption in \S 2.1 that $e_{0,i}\leq e_{0,j}$.

\medskip

In the following we prove some general facts about cyclic extensions which will be used later.

\begin{lemma}\label{bicyclic galois group}
Let $M/k$ and $N/k$ be cyclic extensions of $p$-power degree with $[N:k]\leq [M:k]$.
Then
$\Gal(MN/k)\simeq\Gal(M/k)\times \Gal(N/ N\cap M)$.
\end{lemma}

\begin{proof}
The nature injection
$\Gal(MN/k) \to  \Gal(M/k)\times \Gal(N/k)$ shows that each element of $\Gal(MN/k)$ has order at most $[M:k]$.
Choose an element in $\Gal(MN/k)$ which projects a generator of $\Gal(M/k)$.
Then it generates a subgroup isomorphic to $\Gal(M/k)$.
Hence the exact sequence
$$\xymatrix@C=0.5cm{
  1 \ar[r] & \Gal(MN/M) \ar[r] & \Gal(MN/k) \ar[r] &  \Gal(M/k) \ar[r] & 1 }$$  splits.
Note that $\Gal(MN/M)$ is isomorphic to $\Gal(N/N\cap M)$.
Therefore $\Gal(MN/k)\simeq \Gal(M/k)\times \Gal(N/N\cap M)$.
\end{proof}

\medskip

\begin{lemma}\label{locally cyclic condition}
Let $M/k$, $N/k$, and $R/k$ be cyclic extensions of $p$-power degree and $v\in\Omega_k$.
Suppose the following:
 \begin{enumerate}
   \item $RM=NM.$
   \item $RN\subseteq RM$.
   \item $RN$ is locally cyclic at $v$, i.e. $RN\otimes_k {k_v}$ is a product of cyclic extensions of $k_v$.
 \end{enumerate}
 Then  either $R^v\otimes_{k_v} N^v$ is isomorphic to a product of copies of $N^v$ or $R^v\otimes_{k_v} M^v$ is isomorphic to a product of copies of $M^v$.
\end{lemma}

\begin{proof}
Let $\tilde M$, $\tilde N$ and $\tilde R$ be cyclic extensions of $k_v$ such that $M^v\simeq \prod \tilde M$,
$N^v\simeq \prod \tilde N$ and $R^v\simeq \prod \tilde R$.

Suppose that $R^v\otimes_{k_v} N^v=\prod \tilde{R}\otimes_{k_v}\tilde{N}\ncong \prod N^v$.
Then $\tilde{R}\cap\tilde {N}\neq \tilde {R}$.
We claim that $\tilde{R}\cap\tilde {N}=\tilde {N}$.
Suppose not. Then $\tilde{R}\tilde{N}$ is a bicyclic extension of $k_v$ and $\tilde{R}\otimes_{k_v}\tilde{N}$ is a product of bicyclic extensions.
As there is a surjective map from $R^v\otimes_{k_v} N^v$ to $RN\otimes_k k_v$  and
by assumption the latter is a product of cyclic extensions, the algebra $R^v\otimes_{k_v} N^v=\prod \tilde{R}\otimes_{k_v}\tilde{N}$ is also a product of cyclic extensions, which is a contradiction.
Hence $\tilde N$ is a proper subfield of $\tilde R$.

Now consider the fields $F_R=\tilde M \cap \tilde R$ and $F_N=\tilde M\cap\tilde N$.
As $RM=NM$, we have $\tilde R \tilde M=\tilde N \tilde M$.
Therefore $[\tilde R \tilde M:\tilde M]=[\tilde R:F_R]=[\tilde N:F_N]$.

We claim that  $\tilde N=F_N$.
Suppose not, i.e. $F_N\varsubsetneq \tilde N$.
Then we have $\tilde N \nsubseteq F_R$.
As they are both subfields of $\tilde R$, which is cyclic of $p$-power degree, this implies that $F_R\subseteq \tilde N$.
Hence $F_N=F_R$.
As  $[\tilde R:F_R]=[\tilde N:F_N]$, we have $\tilde R=\tilde N$, which is a contradiction.
Hence $F_N=\tilde N$ and $[\tilde R:F_R]=[\tilde N:F_N]=1$.
Since $\tilde R=F_R\subseteq \tilde M$, the algebra $R^v\otimes_{k_v} M^v$ is isomorphic to a product of copies of $M^v$.
\end{proof}

\medskip

\begin{lemma}\label{subfield of bicyclic}
Let $i,j\in \cI'$ and $i\neq j$. Let $R$ be a cyclic extension of $k$ of degree $p^d$.
Set $F=K_i\cap K_j\cap R$ and $p^h=[F:k]$.
Suppose that $R\subseteq K_i K_j$ and $d\leq \min\{\epsilon_i,\epsilon_j\}$. Then  ${d+e_{i,j}-h}\leq\min\{\epsilon_i, \epsilon_j\}$ and $R\subseteq K_i({d+e_{i,j}-h})K_j({d+e_{i,j}-h})$.
\end{lemma}

\begin{proof}
By the definition of $h$, we have $h\leq e_{i,j}$.
If $h=e_{i,j}$, then ${d+e_{i,j}-h}\leq\min\{\epsilon_i, \epsilon_j\}$ by assumption.
If $h<e_{i,j}$, we claim that $R\cap K_i=R\cap K_j=F$.
To see this, first note that $R\cap K_i$ and $K_j\cap K_i$ are both subfields of the cyclic extension $K_i$.
Hence either $R\cap K_i\subseteq K_j\cap K_i$ or $K_j\cap K_i\varsubsetneq R\cap K_i$.
If $K_j\cap K_i\varsubsetneq R\cap K_i$, then $F=K_j\cap K_i$ which contradicts to the assumption $h<e_{i,j}$.
Therefore $R\cap K_i\subseteq K_j\cap K_i$. This implies that $R\cap K_i=F$.
Similarly we have $R\cap K_j=F$.

As $RK_i\subseteq K_iK_j$, by comparing the degrees of both sides, we have $\epsilon_i+d-h\leq \epsilon_i+\epsilon_j-e_{i,j}$ and hence ${d+e_{i,j}-h}\leq \epsilon_j$.
One can get ${d+e_{i,j}-h}\leq \epsilon_i$ by a similar way.

Next we show the second part of the statement.
If $h=d$, then by definition $R=K_i(d)=K_j(d)$ and the lemma is clear.
Suppose $h<d$.
We regard $R$, $K_i$ and $K_j$ as extensions of $F$.
Let $M=K_i({d+e_{i,j}-h})K_j({d+e_{i,j}-h})$.
Without loss of generality, we assume $\epsilon_i\geq\epsilon_j$.
By Lemma \ref{bicyclic galois group}, the Galois group $\Gal(K_iK_j/F)$ is isomorphic to $\Gal(K_i/F)\times\Gal(K_j/K_i\cap K_j)\simeq\ent/p^{\epsilon_i-h}\ent\times\ent/p^{\epsilon_j-e_{i,j}}$.
Let $(a,b)\in\ent/p^{\epsilon_i-h}\ent\times\ent/p^{\epsilon_j-e_{i,j}}$.
If $(a,b)$ fixes $M$, then $a$ fixes $K_i({d+e_{i,j}-h})$ and $b$ fixes $K_j({d+e_{i,j}-h})$. Hence there are $x$ and $y$ such that $a=p^{d+e_{i,j}-2h}x$ and $b=p^{d-h}y$.

On the other hand $R$ is a cyclic extension of degree $p^{d-h}$ of $F$, so
for every $\sigma\in\Gal(K_iK_j/F)$, we have ${p^{d-h}}\sigma\in\Gal(K_iK_j/R)$.
Hence we have $\Gal(K_iK_j/M)\subseteq \Gal(K_iK_j/R)$ and $R\subseteq M$.

\end{proof}

\medskip

For a nonempty subset $C\subseteq \cI$ and an integer $d\geq 0$, we define the field $M_{C}(d)$ to be the composite field $\langle K_i(d)\rangle_{i\in C}$.
\begin{lemma}\label{generator of bicyclic}
Let $d$ be a positive integer and $J$ be a non-empty subset of $\cI'$.
Suppose that $M=M_J(d)$ is bicyclic. Then
$M=K_i(d)K_j(d)$, for any  $i,j\in J$  such that the degree of $K_i(d)K_j(d)$ is maximal.
\end{lemma}

\begin{proof}
As $M$ is bicyclic, there are at least two elements in $J$.
If $|J|=2$, then the claim is trivial.

Suppose that $|J|>2$.
Pick $i,j\in J$ such that the degree of $K_i(d)K_j(d)$ is maximal.
If $d\leq e_{i,j}$, then  $K_i(d)K_j(d)=K_i(d)$ which is of degree $p^d$.
Since for any $s\in J$ the degree of $K_i(d)K_s(d)$ is at least $p^d$, we have $K_i(d)K_s(d)=K_i(d)$ for all $s\in J$. Hence $M=K_i(d)$ which is a cyclic extension. This contradicts to our assumption.
Therefore $d>e_{i,j}$.

We claim that for any $s\in J$, the field $K_s(d)$ is contained in $K_i(d)K_j(d)$.
As the degree of $K_i(d)K_j(d)$ is maximal, the degree of $K_i(d)\cap K_j(d)$ is minimal.
Since $K_i$ is cyclic, this implies that $K_i(d)\cap K_j(d)\subseteq K_i(d)\cap K_s(d)$.
Set  $N=K_i(d)K_j(d)\cap K_s(d)$.
Note that $N$ is a cyclic extension. Let  $p^l$ be the degree of $[N:k]$.

We claim that $N=K_s(d)$.
Suppose that $N\varsubsetneq K_s(d)$, i.e. $l<d$.
Then  $K_i(d)K_j(d)$ is a bicyclic extension of $K_i(l)K_j(l)$.
Since $K_i(d)\cap K_j(d)\subseteq K_i(d)\cap K_s(d)$, we have $K_i(d)\cap K_j(d)\subseteq N$.
By Lemma \ref{subfield of bicyclic} we have $N\subseteq K_i({l})K_j({l})$.
Therefore $K_i(d)K_j(d)/N$ is a bicyclic extension of $N$.

Note that $\Gal(K_i(d)K_j(d)K_s(d)/N)\simeq \Gal(K_i(d)K_j(d)/N)\times \Gal(K_s(d)/N)$.
Since $N\varsubsetneq K_s(d)$ and $\Gal(K_i(d)K_j(d)/N)$ is bicyclic, the field $K_i(d)K_j(d)K_s(d)$ is not a bicyclic extension of $k$, which contradicts to the fact that $M$ is a bicyclic extension.
Hence $N=K_s(d)$ and $K_s(d)\subseteq K_i(d)K_j(d)$.
\end{proof}

\medskip

\begin{lemma}\label{definition of a'}
Let $a=(a_1,...,a_m)$ be an  element in $G_{\omega}(K_0,K')\setminus D$.
Set $\epsilon_0-d=\underset{i\notin I_{a_1}(a)}{\min}\{\delta(a_1,a_i)\}$.
Choose $j\notin I_{a_1}(a)$ minimal such that $\epsilon_0-d=\delta(a_1,a_j)$.
Set $a'=(a'_1,...,a'_m)\in \underset{i\in \cI}{\oplus}\ent/p^{e_i}\ent$ as follows:
\begin{equation}
a'_i=\left\{
\begin{array}{ll}
\pi_{e_j,e_i}(a_j), & \hbox{if $i\notin I_{a_1}(a)$ and $\epsilon_0-d=\delta(a_1,a_i)$};\\
\pi_{e_1,e_i}(a_1), & \hbox{otherwise.}
\end{array}
\right .
\end{equation}
Then $a'\notin D$ and $\cS_{a'}\subseteq \cS_{a}$.
\end{lemma}

\begin{proof}
First note that $d> e_{0,i}$ for all $i\notin I_{a_1}(a)$.
As $j\notin I_{a_1}(a')$, we have $a'\notin D$.

The inclusion $\cS_{a'}\subseteq \cS_{a}$ is equivalent to the inclusion
 $\underset{n\in\ent/p^{e_1}\ent}{\bigcup} \Omega (I_n(a))\subseteq\underset{n\in\ent/p^{e_1}\ent}{\bigcup} \Omega (I_n(a'))$, i.e. for $n\in\ent/p^{e_1}\ent$ and for $v\in\Omega (I_n(a))$, there is some $n'\in \ent/p^{e_1}\ent$ such that $v\in\Omega (I_{n'}(a'))$.
It is enough to show that for each $n\in \ent/p^{e_1}\ent$, there is some $n'\in \ent/p^{e_1}\ent$ such that $I_{n}(a)\subseteq I_{n'}(a')$ and $\delta(n,a_i)\leq\delta (n',a'_i)$ for all $i\notin I_{n'}(a')$.

\noindent
\textbf{case 1.} $\delta(a_1,n)> \epsilon_0-d$.
We claim that $I_n(a)\subseteq I_{a_1}(a')$ in this case.
For all $i\in I_n(a)$, we have $\delta(a_1,a_i)=\delta(a_1,\pi_{e_1,e_i}(n))=\min\{\delta(a_1,n),e_i\}$.
Hence we have either $\delta(a_1,a_i)=\delta(a_1,n)>\epsilon_0-d$ or $i\in I_{a_1}(a)$.
Therefore $a'_i=\pi_{e_1,e_i}(a_1)$ and $i\in I_{a_1}(a')$.

By the construction of $a'$, for any $i\notin I_{a_1}(a')$ we have $\delta(a_1,a_i)=\epsilon_0-d$ and $\delta(a'_1,a'_i)=\delta(a_1,a_j)=\epsilon_0-d$.
Since $\delta(a_1,n)>\epsilon_0-d$ and $\delta(a_1,a_i)=\epsilon_0-d$,
we have $\delta(n,a_i)=\epsilon_0-d=\delta(a'_1,a'_i)$.

\noindent
\textbf{case 2.} $\delta(a_1,n)=\epsilon_0-d$.
Then  for all $i\in I_n(a)\setminus I_{a_1}(a)$, we have $\delta(a_1,a_i)=\delta(a_1,n)=\epsilon_0-d$. If $i\in I_n(a)\cap I_{a_1}(a)$, then $e_i\leq \epsilon_0-d$ and hence
$\pi_{e_1,e_i}(a_1)=\pi_{e_j,e_i}(a_j)$.
In both cases, we have
$a'_i=\pi_{{e_j},{e_i}}(a_j)$  and $i\in I_{n'}(a')$ for any $n'\in \ent/p^{e_1}\ent$ such that $a_j=\pi_{e_1,{e_j}}(n')$.

Let $i\notin I_{n'}(a')$. Then we have $e_i>\epsilon_0-d$ and $a'_i=\pi_{e_1,{e_i}}(a_1)$.
This implies $\delta(n',a'_i)=\delta(a_j,a_1)=\epsilon_0-d$.
On the other hand $\delta(a_1,a_i)>\epsilon_0-d$ for any $i\notin I_{n'}(a')$.
Hence $\delta(n,a_i)=\delta(n,a_1)=\epsilon_0-d$ and $\delta(n',a'_i)=\delta(n,a_i).$

\noindent
\textbf{case 3.} $\delta(a_1,n)<\epsilon_0-d$.
Since $\epsilon_0-d=\underset{i\notin I_{a_1}(a)}{\min}\{\delta(a_1,a_i)\}$,
we have $I_{n}(a)\subseteq I_{a_1}(a)\subseteq I_{a_1}(a')$.
For $i\notin I_{a_1}(a')$, we have $\delta(a_1,a_i)=\epsilon_0-d$ and hence $\delta(n,a_i)=\delta(n,a_1)<\epsilon_0-d=\delta(a_1,a'_i)$.

From the above three cases we conclude that $\cS_{a'}\subseteq \cS_{a}$.
\end{proof}

We immediately have the following corollary.
\begin{coro}\label{step of a}
Keep notation as above.
If $a$  $\in G(K_0,K)\backslash D$ (resp. $G_\omega(K_0,K')\setminus D$), then $a'\in G(K_0,K')\backslash D$. (resp. $G_\omega(K_0,K')\setminus D$).
\end{coro}

\begin{lemma}\label{K_0(d) splits}
Let $a=(a_1,...,a_m)$ be an  element in $G_\omega(K_0,K')\setminus D$.
Set $\epsilon_0-d=\underset{i\notin I_{a_1}(a)}{\min}\{\delta(a_1,a_i)\}$.
Choose $s,t\in I$ such that $\delta(a_1,a_s)>\epsilon_0-d$ and $\delta(a_1,a_t)= \epsilon_0-d$.
Then there is a finite set $\cS\subseteq\Omega_k$ such that for all $v\in\Omega_k\setminus \cS$
either $K_0(d)\otimes K^v_s$ is a product of copies of $K^v_s$ or  $K_0(d)\otimes K^v_t$ is a product of copies of $K^v_t$.
Moreover, if $a\in G(K_0,K')\setminus D$, then we can take $\cS=\emptyset$.
\end{lemma}

\begin{proof}
Let $a'$ be defined as in Lemma \ref{definition of a'}.
Then $a'\in G_\omega(K_0,K')\setminus D$.
We claim that for all $v\in\Omega_k\setminus\cS_{a'}$ either $K_0(d)\otimes K^v_s$ is a product of copies of $K^v_s$ or  $K_0(d)\otimes K^v_t$ is a product of copies of $K^v_t$.
Note that if $a\in G(K_0,K')\setminus D$, then $\cS_{a'}=\emptyset$.

Let $v\in\Omega_k\setminus \cS_{a'}$.
By the definition of $\cS_{a'}$, there is $n\in\ent/p^{e_1}\ent$ such that $v\in\Omega(I_n(a'))$.
We consider the following cases.

\noindent
\textbf{case 1:} $\delta(a_1,n)\leq \epsilon_0-d$.
Then $s\notin I_n(a')$ and $\delta(a_s,n)\leq \epsilon_0-d$.
By the definition of $\Omega(I_n(a'))$, we have $v\in\Sigma_s^{\epsilon_0-d}$.
Hence $K_0(d)\otimes K^v_s$ is a product of copies of $K^v_s$.

\noindent
\textbf{case 2:} $\delta(a_1,n)> \epsilon_0-d$.
If $t\in I_{a_1}(a')$, then $e_t=\epsilon_0-d$.
Hence $K_0(d)\otimes K^v_t$ is a product of copies of $K^v_t$.

Suppose that  $t\notin I_{a_1}(a')$.
Then $t\notin I_n(a')$ and $\delta(a_t,n)=\epsilon_0-d$.
By the definition of $\Omega(I_n(a'))$, we have $v\in\Sigma_t^{\epsilon_0-d}$.
Hence $K_0(d)\otimes K^v_t$ is a product of copies of $K^v_t$.
\end{proof}

\begin{prop}\label{bicyclic field}
Let $a=(a_1,...,a_m)$ be an  element in $G_\omega(K_0,K')\setminus D$.
Set $\epsilon_0-d=\underset{i\notin I_{a_1}(a)}{\min}\{\delta(a_1,a_i)\}$.
For any $s,t\in I$ with $\delta(a_1,a_s)>\epsilon_0-d$ and $\delta(a_1,a_t)= \epsilon_0-d$, we set
$u= \max\{s,t\}$. Let $\beta=\min\{e_{0,s},e_{0,t}\}$.
Then we have the following:
\begin{enumerate}
  \item The extension $K_0(d)\subseteq F_{d,s,t}:=K_s({d+e_{s,t}-\beta})K_t({d+e_{s,t}-\beta})$.
  Moreover, if  $e_{0,s}=e_{0,t}$, then $F_{d,s,t}=K_0(d)K_s({d+e_{s,t}-\beta})=K_0(d)K_t({d+e_{s,t}-\beta}).$
  \item Suppose further that $a\in G(K_0,K')$. Then the field $K_0(d)K_u({d+e_{s,t}-\beta})$ is locally cyclic
\end{enumerate}
\end{prop}
\begin{proof}

Let $s,t\in I$ as above.
By Lemma \ref{K_0(d) splits}, there is a finite set $\cS$ such that for all $v\in\Omega_k\setminus\cS$ either $K_0(d)\otimes K^v_s$ is a product of copies of $K^v_s$ or  $K_0(d)\otimes K^v_t$ is a product of copies of $K^v_t$.
Hence $K_0(d)\otimes (K_sK_t)^v$ is a product of  copies of $(K_sK_t)^v$ for all $v\notin \cS$.

Since $\cS$ is a finite set, by Chebotarev's density theorem $K_0(d)\subseteq K_sK_t$.
Since $\epsilon_0-e_{0,s}=e_s\geq\delta(a_1,a_s)>\epsilon_0-d$, we have $d>\beta$.
By Lemma \ref{subfield of bicyclic}, we have $K_0(d)\subseteq K_s({d+e_{s,t}-\beta})K_t({d+e_{s,t}-\beta})$.

If $e_{0,s}=e_{0,t}$, then $\beta=e_{0,s}=e_{0,t}$.
By dimension reasons we have $K_0(d)K_s({d+e_{s,t}-\beta})=F_{d,s,t}=K_0(d)K_t({d+e_{s,t}-\beta})$.
This proves the first statement.

Suppose that $a\in G(K_0,K')$.
Since $d>\beta$, the field $F_{d,s,t}$ is bicyclic,
and its Galois group is isomorphic to $\ent/p^{d+e_{s,t}-\beta}\ent\times\ent/ p^{d-\beta}\ent$ by Lemma \ref{bicyclic galois group}.

We first assume that $e_{0,s}\leq e_{0,t}$. Then $\beta=e_{0,s}$ and $F_{d,s,t}=K_0(d)K_s({d+e_{s,t}-\beta})$ by  dimension reasons.
Note that  the field $K_0(d)K_t({d+e_{s,t}-\beta})$ is contained in $F_{d,s,t}$.
By Lemma \ref{K_0(d) splits}, at each place $v\in\Omega_k$ we have either $K_0(d)\otimes_k K^v_s$ splits into a product of $K^v_s$ or $K_0(d)\otimes_k K^v_t$ splits into a product of $K^v_t$.
For a place $v\in\Omega_k$, if $K_0(d)\otimes_k K^v_s$ splits into a product of $K^v_s$, then $F^v_{d,s,t}$ is a product of cyclic extensions of $k_v$. As a subalgebra of $F^v_{d,s,t}$, the algebra $(K_0(d)K_t({d+e_{s,t}-\beta}))^v$ is a product of cyclic extensions.
If  $K_0(d)\otimes_k K^v_t$ splits into a product of $K^v_t$, then $(K_0(d)K_t({d+e_{s,t}-\beta}))^v$ is a product of cyclic extensions.
Hence $K_0(d)K_t({d+e_{s,t}-\beta})$ is locally cyclic.

For $e_{0,s}\geq e_{0,t}$, a similar argument works.

\end{proof}

\medskip

\section{Patchable subgroups}

Recall that for each nonempty subset $U_r$ we define $G_\omega(K_0,K_{U_r})$ and there is a natural projection from
$G_\omega(K_0,K')$ to $G_\omega(K_0,K_{U_r})$. (See \S 2.3 for details.)
In view of the combinatorial description of $\sha^2_\omega(k,\hat{T}_{L/k})$ (resp. $\sha^2(k,\hat{T}_{L/k})$), the computation of $\sha^2_\omega(k,\hat{T}_{L/k})$ (resp. $\sha^2(k,\hat{T}_{L/k})$) will be much simpler if
the $e_i$'s are equal.
Hence we will calculate $G_\omega(K_0,K_{U_r})$  for each nonempty subset $U_r$  and then
"patch" them together to get the group $G_\omega(K_0,K')$.

Suppose that an element $x\in G_\omega(K_0,K_{U_r})$ can be patched into an element in
$ G_\omega(K_0,K')$. Then $x$ must be in the image of $G_\omega(K_0,K')$ under the projection map $\varpi_r$.
(See Section 2.3 for the definition of $\varpi_r$.)

Let $G^0_\omega(K_0,K')$  (resp. $G^0(K_0,K')$) be the subgroup consisting of elements $(a_1,...,a_m)\in G_\omega(K_0,K')$ (resp. $G(K_0,K')$) with $a_1=0$.
Then $G_\omega(K_0,K')=D\oplus G^0_\omega(K_0,K')$.
In Section 4.1 we define the \emph{patchable subgroup} $\tilde{G}(K_0,K_{U_r})$, which is in fact
the image of $G^0_\omega(K_0,K')$ under the projection map $\varpi_r$.

We show that there is a section of $\varpi_r$ from $\tilde{G}_\omega(K_0,K_{U_r})$ to $G^0_\omega(K_0,K_{U_r})$ and prove that $G_\omega(K_0,K')=D\oplus\underset{r}{\oplus}\tilde{G}_\omega(K_0,K_{U_r})$, where $r$ runs over positive integers with $U_r$ nonempty.
We prove similar results for $G(K_0,K')$.

Note that if $\cI=U_r$ for some $r$, then $G_\omega(K_0,K')=G_\omega(K_0,K_{U_r})$ and
no patching condition is needed.
Hence in the following we fix an integer $r$ such that $U_r$ is not empty and $U_r\neq \cI$.

We set $U_{>r}=\{i\in\cI|e_{0,i}>r\}$ and $U_{<r}=\{i\in\cI|e_{0,i}<r\}$.
Recall that we assume $\underset{i\in I'}{\cap} K_i=k$.
Hence $U_0$ is nonempty.

Recall that for a nonempty subset $C\subseteq \cI$ and an integer $d\geq 0$, $M_{C}(d)$ is the composite field $\langle K_i(d)\rangle_{i\in C}$.


\subsection{Algebraic patching degrees}

\begin{defn}\label{alg patching degree}
  Define the \emph{algebraic patching degree} $\Delta^\omega_r$ of $U_r$ to be the maximum nonnegative integer $d$
satisfying the following:
\begin{enumerate}
  \item If $U_{>r}$ is nonempty, then $M_{U_{>r}}(d)\subseteq\underset{i\in U_{r}}{\cap} K_0(d)K_i(d)$.
  \item If $U_{<r}$ is nonempty, then $M_{U_r}(d)\subseteq \underset{i\in U_{<r}}{\cap} K_0(d)K_i(d)$.
\end{enumerate}
If $U_r=\cI$, then we set $\Delta^\omega_r=\epsilon_0$.
\end{defn}

Note that  $K_0(r)=K_i(r)$ for all $i\in U_{\geq r}$. Hence by definition we have $\Delta^\omega_r\geq r$.
By Lemma \ref{subfield of bicyclic} and the definition of $\Delta^\omega_r$,  all nonnegative integers $d\leq\Delta^\omega_r$ satisfy the conditions (1) and (2) in the above definition.

\begin{lemma}\label{coro from def of Delta omega}
Let $d\leq\Delta^\omega_r$ be a nonnegative integer. If $U_{<r}$ is nonempty, then $K_0(d)K_j(d)\subseteq\underset{i\in U_{\leq r}}{\cap}K_0({d})K_i({d})$ for all $j\in U_r$
\end{lemma}

\begin{proof}

Suppose that $U_{<r}$ is nonempty.
If $d\leq r$, the claim is trivial.
Assume $d> r$.
By the definition of $\Delta_r^\omega$ the field $K_0({\Delta^\omega_r})M_{U_r}(\Delta^\omega_r)$ is contained in $K_0({\Delta^\omega_r})K_i({\Delta^\omega_r})$ for all $i\in U_{<r}$.
By Lemma \ref{generator of bicyclic} we have $K_0(d)M_{U_r}(d)=K_0(d)K_j(d)$
for all $j\in U_r$.
By Lemma \ref{subfield of bicyclic} we have  $K_0(d)K_j(d)\subseteq \underset{i\in U_{< r}}{\cap}K_0(d)K_i(d)$ for all $j\in U_r$.
Hence $K_0(d)K_j(d)\subseteq \underset{i\in U_{\leq r}}{\cap}K_0(d)K_i(d)$ for all $j\in U_r$.
\end{proof}

\begin{prop}\label{bound on Delta omega}
Suppose that $U_{>r}$ is nonempty.
Let $r'$ be the smallest positive integer bigger than $r$ such that $U_{r'}$ is nonempty.
Then we have the following:
\begin{enumerate}
 \item If $r=0$, then $\Delta^\omega_r=\Delta^\omega_{r'}$.
 \item $\Delta^\omega_r\leq\Delta^\omega_{r'}$.
 \item $\Delta^\omega_r-r\geq\Delta^\omega_{r'}-r'$.
\end{enumerate}

\end{prop}

\begin{proof}
We first show (2).
Note that by our choice of $r'$, we have $U_{<r'}=U_{\leq r}$, which is nonempty.
By the definition of $\Delta^\omega_r$ and by Lemma \ref{coro from def of Delta omega}, we have
$M_{U_{r'}}(\Delta^\omega_r)\subseteq \underset{i\in U_{r}}{\cap} K_0({\Delta^\omega_r})K_i({\Delta^\omega_r})\subseteq \underset{i\in U_{< r'}}{\cap} K_0({\Delta^\omega_r})K_i({\Delta^\omega_r})$.

Suppose that $U_{>r'}$ is nonempty.
As $M_{U_{>r}}(\Delta^\omega_r)\subseteq K_0({\Delta^\omega_r})K_i({\Delta^\omega_r})$ for all $i\in U_r$,
 by Lemma \ref{generator of bicyclic} we have $K_0({\Delta^\omega_r})M_{U_{>r}}(\Delta^\omega_r)= K_0({\Delta^\omega_r})K_j({\Delta^\omega_r})$ for all $j\in U_{r'}$.
Hence $M_{U_{>r'}}(\Delta^\omega_r)\subseteq\underset{i\in U_{r'}}{\cap} K_0({\Delta^\omega_r})K_i({\Delta^\omega_r})$ and $\Delta^\omega_{r'}\geq\Delta^\omega_r$.

Suppose that $r=0$. Then $U_{<r'}=U_0$.
By the definition of $\Delta_{r'}^\omega$,  we have $M_{U_{> r'}}(\Delta^\omega_{r'})\subseteq\underset{i\in U_{ r'}}{\cap}K_0({\Delta^\omega_{r'}})K_i({\Delta^\omega_{r'}})$ and
 $K_0(\Delta^\omega_{r'})K_j(\Delta^\omega_{r'})$ is contained in
$\underset{i\in U_{0}}{\cap}K_0({\Delta^\omega_{r'}})K_i({\Delta^\omega_{r'}})$ for all $j\in U_{r'}$.
Hence $M_{U_{\geq r'}}(\Delta^\omega_{r'})\subseteq\underset{i\in U_{ 0}}{\cap}K_0({\Delta^\omega_{r'}})K_i({\Delta^\omega_{r'}})$.
Therefore $\Delta^\omega_{r'}\leq\Delta^\omega_0$. Combining this with statement (2), we get  (1).

Now suppose that $r>0$. We claim that $\Delta^\omega_{r'}-r'\leq\Delta^\omega_r-r$.
By  Lemma \ref{coro from def of Delta omega}, we have
$K_0({\Delta^\omega_{r}})K_i({\Delta^\omega_r})\subseteq\underset{j\in U_{<r'}}{\cap}K_0({\Delta^\omega_{r}})K_j({\Delta^\omega_r})$ for all $i\in U_r$.

Let $F=\underset{i\in U_{< r'}}{\cap}K_0({\Delta^\omega_{r'}})K_i({\Delta^\omega_{r'}})$.
By Lemma \ref{coro from def of Delta omega} we have $K_0({\Delta^\omega_{r'}})K_i({\Delta^\omega_r})\subseteq F,$
for all $i\in U_r$.

Let $i\in U_r$.
As $K_0({\Delta^\omega_{r'}})K_i({\Delta^\omega_r})\subseteq F\subseteq K_0({\Delta^\omega_{r'}})K_i({\Delta^\omega_{r'}})$, there is some $\Delta^\omega_r\leq\gamma\leq \Delta^\omega_{r'}$ such that $F=K_0({\Delta^\omega_{r'}})K_i({\gamma})$. As $i\in U_r$, the field $F$ is a cyclic extension of $K_0({\Delta^\omega_{r'}})$ of degree $p^{\gamma-r}$.
By the definition of $\Delta^\omega_{r'}$, for all $j\in U_{r'}$ we have $K_0({\Delta^\omega_{r'}})K_j({\Delta^\omega_{r'}})\subseteq F$ and $K_0({\Delta^\omega_{r'}})K_j({\Delta^\omega_{r'}})$ is a cyclic extension of $K_0({\Delta^\omega_{r'}})$ of degree $\Delta^\omega_{r'}-r'$. Hence $\Delta^\omega_{r'}-r'\leq \gamma-r$ for dimension reasons.

Suppose that $\Delta^\omega_{r'}-r'>\Delta^\omega_r-r$. Then $\gamma-r\geq\Delta^\omega_{r'}-r'\geq\Delta^\omega_r+1-r$.
For dimension reasons $K_0({\Delta^\omega_{r'}})K_i({\Delta^\omega_{r}+1})\subseteq F$.
Since $K_i({\Delta^\omega_r+1})\subseteq F\subseteq K_0({\Delta^\omega_{r'}})K_j({\Delta^\omega_{r'}})$ for all $j\in U_{<r'}$, by Lemma \ref{subfield of bicyclic} we have $K_i({\Delta^\omega_r+1})\subseteq K_0({\Delta^\omega_r+1})K_j({\Delta^\omega_r+1})$ for all  $j\in U_{<r}$.
Hence $\Delta^\omega_r+1$ satisfies condition (2) in Definition \ref{alg patching degree}.

By the choice of $r'$ and the definition of $\Delta^\omega_{r'}$, we have $U_{>r}=U_{\geq r'}$ and
$M_{U_{\geq r'}}(\Delta^\omega_{r'})\subseteq
\underset{i\in U_{<r'}}{\cap}K_0({\Delta^\omega_{r'}})K_i({\Delta^\omega_{r'}})\subseteq\underset{i\in U_{r}}{\cap}K_0({\Delta^\omega_{r'}})K_i({\Delta^\omega_{r'}})$.
Thus $\Delta^\omega_{r'}$ satisfies condition (1) in Definition \ref{alg patching degree}.
By assumption $\Delta^\omega_{r'}>\Delta^\omega_{r}+1$, so we have
$\Delta^\omega_r\geq\Delta^\omega_r+1$, which is a contradiction.
Therefore $\Delta^\omega_{r'}-r'\leq \Delta^\omega_r-r$. This proves statement (3).

\end{proof}

\begin{defn}\label{alg patchable}
Suppose that $U_r$ is nonempty.
Let $x=(x_i)_{i\in U_r} \in G_\omega(K_0,K_{U_r})$. We say that $x$ is \emph{algebraically patchable} if  $\delta(0,x_i)\geq\epsilon_0-\Delta^\omega_r$ for all $i\in U_r$. Here we regard $0$ as an element in $\ent/p^{\epsilon_0}\ent$. We define \emph{the algebraic patchable subgroup} of $G_\omega(K_0,K_{U_r})$ as follows:
If $r> 0$, it is the subgroup consisting of all algebraically patchable elements of $G_\omega(K_0,K_{U_r})$;
if $r=0$, it is the subgroup consisting of all algebraically patchable elements of $G_\omega(K_0,K_{U_0})$ with $x_1=0$.
\end{defn}

For $x=(x_i)_{i\in U_r}\in G_\omega(K_0,K_{U_r})$,
define $a_{x}=(a_1,...,a_m)\in \underset{i\in \cI}{\oplus}\ent/p^{e_i}\ent$ as follows:
\begin{equation}\label{definition of ax}
a_i=\left\{
\begin{array}{ll}
x_i, & \hbox{if $i\in U_r,$}\\
0, & \hbox{otherwise.}
\end{array}
\right .
\end{equation}

In the following we show that $x$ is algebraically patchable if and only if $a_x$  is
in $G^0_\omega(K_0,K')$.

\begin{prop}\label{projector alg_pachable}
Let $x\in G_\omega(K_0,K_{U_r})$ and $a_{x}$ be defined as above.
If $a_{x}\in G^0_\omega(K_0,K')$, then $x$ is algebraically patchable.
\end{prop}
We first prove the following Lemma.

\begin{lemma}\label{M_U equal M_U'}
Keep the notation as in Proposition \ref{projector alg_pachable}.
Suppose that $a_{x}=(a_1,...,a_m)\in G_\omega(K_0,K')\setminus D$.
Set $\epsilon_0-d=\underset{i\notin I_{a_1}(a_x)}{\min}\{\delta (a_1,a_i)\}$ and
$U_r'=\{i\in U_r|\epsilon_0-d=\delta(a_1,a_i)\}.$
If $U'_r$ and $U_r\setminus U_r'$ are both nonempty, then $K_0({d})K_s({d})=K_0({d})K_t({d})$ for any $s\in U_r\setminus U'_r$ and any $t\in U_r'$.
In particular $K_0({d})M_{U_r}({d})=K_0({d})M_{U'_r}({d})=K_0({d})M_{U\setminus U'_r}({d})$.
\end{lemma}
\begin{proof}
Suppose  that $U'_r$ is nonempty and $U'_r \varsubsetneq U_r$. Let $t\in U'_r$ and $i\in U_r\setminus U_r'$.
By Proposition \ref{bicyclic field}, we have $K_0({d})\subseteq K_i({d+e_{i,t}-r})K_t({d+e_{i,t}-r})$, and
$K_0({d})K_i({d+e_{i,t}-r})=K_0({d})K_t({d+e_{i,t}-r})$.
Regard $K_0({d})K_t({d+e_{i,t}-r})$ as a cyclic extension of $K_0({d})$.
Then $K_0({d})K_i({d})$ and $K_0({d})K_t({d})$ are subfields of the same degree of the cyclic extension $K_0({d})K_t({d+e_{i,t}-r})$. Hence $K_0({d})K_i({d})=K_0({d})K_t({d})$ for all $t\in U'_r$ and all $i\in U_r\setminus U_r'$. As a consequence $K_0({d})M_{U_r}(d)=K_0({d})M_{U'_r}({d})=K_0({d})M_{U\setminus U'_r}({d})$.
\end{proof}

\medskip

\begin{proof}[Proof of Proposition \ref{projector alg_pachable}]
Suppose that $a_{x}=(a_1,...,a_m)\in G^0_\omega(K_0,K')$.
If $x=0$, then there is nothing to prove.
Hence in the following we assume $x\neq 0$.
Note that $a_1=0$.
Set  $\epsilon_0-d=\underset{i\in U_r}{\min}\{\delta (0,a_i)\}$
and $U_r'=\{i\in U_r|\epsilon_0-d=\delta(0,a_i)\}.$

Since $x\neq 0$, we have $d>r$.
It is enough to prove that $\Delta^\omega_r\geq d$, i.e. $d$ satisfies conditions (1) and (2) in Definition \ref{alg patching degree}.

Suppose that $U_{<r}$ is nonempty.
For any $s\in U_{<r}$ and $t\in U'_r$, we have $e_{s,t}=e_{0,s}$.
By  Proposition \ref{bicyclic field}, we have $K_0({d})\subseteq K_s({d})K_t({d})$.
For dimension reasons $K_s({d})K_t({d})=K_0({d})K_s({d})$. Hence $K_0({d})K_t({d})\subseteq K_0({d})K_s({d})$.
As $s$ and $t$ are arbitrary, we have $K_0({d})M_{U'_r}(d)\subseteq \underset{s\in U_{<r}}{\cap} K_0({d})K_s({d})$.
If $U_r=U'_r$, then we are done.
If not, then by Lemma \ref{M_U equal M_U'} we have $K_0({d})M_{U_r}(d)\subseteq \underset{s\in U_{<r}}{\cap} K_0({d})K_s({d})$.

Now suppose that $U_{>r}$ is not empty. For $s\in U_{\geq d}$, we have $K_s(d)=K_0(d)$.
Suppose that $U_{>r}\setminus U_{\geq d}$ is not empty. Let $s\in U_{>r}\setminus U_{\geq d}$ and $t\in U'_r$. Then $e_{s,t}=e_{0,t}$ and
 by Proposition \ref{bicyclic field}, we have $K_0({d})\subseteq K_s({d})K_t({d})=K_0({d})K_t({d})$.
Since $s$ and $t$ are arbitrary, by Lemma \ref{M_U equal M_U'} we have $K_0({d})M_{U_{>r}}(d)\subseteq \underset{t\in U_{r}}{\cap} K_0({d})K_t({d})$.
Therefore $\Delta^\omega_r\geq d$ and $x$ is algebraically patchable.
\end{proof}

Let $x\in G_\omega(K_0,K_{U_r})$ and denote by
$(\overline{I}_0,...,\overline{I}_{p^{\epsilon_0-r}-1})$ the partition of $U_r$ defined by $x$.
Recall that $\cS_x$ is the finite subset of $\Omega_k$ such that $\underset{n\in\ent/p^{\epsilon_0-r}\ent}{\bigcup} \Omega (\ol{I}_n)=\Omega_k\setminus \cS_x$. (See \S 2.)

\begin{defn}\label{subset S_r}
Suppose that $U_r$ is nonempty.
For a nonnegative integer $d\leq\Delta^\omega_r$ we define $\cS_{r}(d)$ and $\cS_{>r}(d)$ as follows.
\begin{enumerate}
  \item Suppose that $U_{>r}$ is nonempty.
   Define $\cS_{>r}(d)$ to be the set of places such that $(K_0(d)M_{U_{>r}}(d))^v$ is not locally cyclic. If $U_{>r}$ is empty, then set $\cS_{>r}(d)=\emptyset$.
  \item  Define  $\cS_{r}(d)$ to be the set of places such that $(K_0(d)M_{U_{r}}(d))^v$ is not locally cyclic.

\end{enumerate}

\end{defn}

Clearly $\cS_{>r}(d)$ and $\cS_r$ are finite sets.

\begin{prop}\label{alg_patch x}
Let $x\in \tilde{G}_\omega(K_0,K_{U_r})$ and $a_{x}$ be defined as in equation (\ref{definition of ax}).
Denote by $I(a_x)=(I_0,...,I_{p^{\epsilon_0}-1})$ the partition of $\cI$ defined by $a_x$.
Let  $d\leq\Delta^\omega_r$ be a nonnegative integer such that $x_i=0$ (mod $p^{\epsilon_0-d}$) for all $i\in U_r$.
Let $\cS=\cS_x\cup \cS_r(d)\cup \cS_{>r}(d)$.
Then
$\underset{n\in\ent/p^{\epsilon_0}\ent}{\bigcup} \Omega (I_n)\supseteq\Omega_k\setminus \cS$.
As a consequence $a_x\in G^0_{\omega}(K_0,K')$.
\end{prop}

\begin{proof}
If $d=r$, then clearly $a_{x}=0 \in G^0_\omega(K_0,K')$.
Hence we assume $d>r$.

We claim that $\underset{n\in\ent/p^{\epsilon_0}\ent}{\bigcup} \Omega (I_n)\supseteq\Omega_k\setminus \cS$.
Set $\Omega_S=\Omega_k\setminus \cS$.
If $\Omega_S\subseteq\Omega(I_0)$, then our claim is clear.
Suppose not.
Let $v\in\Omega_S\setminus \Omega(I_0)$.
Our aim is to find $n\neq 0$ such that $v\in\Omega(I_n)$.
Since $v\notin \Omega(I_0)$, there is $t\in U_r\setminus I_0$ such that $v\notin \Sigma_t^{\delta(0,x_t)}$.
As $x_t=0$ (mod $p^{\epsilon_0-d}$), we have $\delta(0,x_t)\geq \epsilon_0-d$.
Then
$K_0(d)^v\otimes_{k_v} K_t(d)^v$ is not a product of copies of $K_t(d)^v$.

Suppose that $U_{>r}$ is nonempty.
By  the choice of $d$ and $\cS$, we have $M_{U_{>r}}(d)\subseteq\underset{i\in U_{r}}{\cap} K_0({d})K_i({d})$ and $(K_0({d})M_{U_{>r}}(d))^v$ is a product of cyclic extensions of $k_v$.
Hence for $s\in U_{>r}$
we have $e_{s,t}=e_{0,t}$
and $K_0(d)K_s(d)\subseteq K_0(d)K_t(d)=K_s(d)K_t(d)$.
As $(K_0({d})K_s(d))^v$ is a product of cyclic extensions of $k_v$, by Lemma \ref{locally cyclic condition}  $K_0(d)^v\otimes_{k_v} K_s(d)^v$ is a product of copies of $K_s(d)^v$.

Suppose that $U_{<r}$ is not empty.
Let $s\in U_{<r}$.
As $K_0(d)K_t(d)\subseteq K_0(d)K_s(d)=K_s(d)K_t(d)$ for all $s\in U_{<r}$, by Lemma \ref{locally cyclic condition}
$K_0(d)^v\otimes_{k_v} K_s(d)^v$ is a product of copies of $K_s(d)^v$.
Hence $v\in \underset{s\notin U_r}\cap\Sigma^{\epsilon_0-d}_{s}$.

Denote  by
$(\overline{I}_0,...,\overline{I}_{p^{\epsilon_0-r}-1})$ the partition of $U_r$ defined by $x$.
By the definition of $\cS_x$, we have $\underset{\overline{n}\in\ent/p^{\epsilon_0-r}\ent}{\bigcup} \Omega (\overline{I}_{\overline{n}})=\Omega_k\setminus \cS_x.$
Hence there is $\overline{n}\neq 0$ such that $v\in \Omega(\overline{I}_{\overline{n}})$. Let $n\in \ent/p^{\epsilon_0}\ent$ such that $\pi_{\epsilon_0,\epsilon_0-r}(n)=\overline{n}$. Then $\delta(n,x_i)=\delta(\overline{n},x_i)$ for all $i\in U_r$.
We claim that $v\in\Omega(I_n)$.

First note that $\delta(n,x_t)>\epsilon_0-d.$ To see this, we first suppose that $t\in I_n$.
Then $\delta(n,x_t)=\epsilon_0-r>\epsilon_0-d$ by the assumption of $d$.
Suppose that $t\notin I_n$. As $v\notin \Sigma_t^{\epsilon_0-d}$ and $v\in\Omega(\ol{I}_{\ol{n}})$, we have $\delta(n,x_t)>\epsilon_0-d$.

As $\delta(n,x_t)>\epsilon_0-d$ and $\delta(0,x_t)\geq\epsilon_0-d$, we have $\delta(n,0)\geq\epsilon_0-d$.
Let $i\notin U_r\cup I_n$.
Then $a_i=0$ and $e_i>\epsilon_0-d$.
Therefore $\delta(n,a_i)=\delta(n,\pi_{e_1,e_i}(0))\geq\epsilon_0-d$ for all $i\notin U_r\cup I_n$.

Since $v\in \underset{s\notin U_r}\cap\Sigma^{\epsilon_0-d}_{s}$ and $\delta(n,a_i)\geq \epsilon_0-d$ for all $i\notin U_r\cup I_n$, we have $v\in \underset{s\notin U_r\cup I_n}\cap\Sigma^{\delta(a_s,n)}_{s}$.
Combining this with the fact that $v\in \Omega({\overline{I}}_{\overline{n}})$, we have $v\in \Omega({I}_{n})$.

As $x\in \tilde{G}_\omega(K_0,K_{U_r})$, the set $S$ is finite.
Hence $a_x\in G^0_\omega(K_0,K')$.
The proposition then follows.

\end{proof}

\subsection{Patching degrees}
\begin{defn}\label{patching degree}
Define the patching degree $\Delta_r$ of $U_r$ to be the maximum nonnegative integer  $d\leq\Delta^\omega_r$
satisfying the following:
\begin{enumerate}
  \item If $U_{>r}$ is nonempty, then  the field $K_0(d)M_{U_{>r}}(d)$ is locally cyclic.
  \item If $U_{<r}$ is nonempty, then the  field $K_0(d)M_{U_{r}}(d)$ is locally cyclic.
\end{enumerate}
If $U_r=\cI$, then we set $\Delta_r=\epsilon_0$.
\end{defn}

Note that  $K_0(r)=K_i(r)$ for all $i\in U_{\geq r}$. Hence by definition we have

From the definition of $\Delta_r$, we have $\Delta^\omega_r\geq\Delta_r\geq r$.

\begin{prop}\label{bound on Delta}
Suppose that $U_{>r}$ is nonempty.
Let $r'$ be the smallest positive integer bigger than $r$ such that $U_{r'}$ is nonempty.
Then we have the following:
\begin{enumerate}
 \item If $r=0$, then $\Delta_r=\Delta_{r'}$.
 \item $\Delta_r\leq\Delta_{r'}$.
 \item $\Delta_{r'}-r'\leq\Delta_{r}-r$.
\end{enumerate}

\end{prop}

\begin{proof}
We first show that $\Delta_r\leq\Delta_{r'}$.
Note that by our choice of $r'$, we have $U_{<r'}=U_{\leq r}$, which is nonempty.
By Proposition \ref{bound on Delta omega} (2), we  $\Delta_r\leq\Delta^\omega_r\leq\Delta^\omega_{r'}$.

Suppose that $U_{>r'}$ is nonempty.
By the definition of $\Delta_r$ the field $K_0({\Delta_r})M_{U_{>r}}(\Delta_r)$ is locally cyclic.
Hence $K_0({\Delta_r})M_{U_{r'}}(\Delta_r)$ and $K_0({\Delta_r})M_{U_{>r'}}(\Delta_r)$ are locally cyclic.
Therefore $\Delta_{r'}\geq\Delta_r$.

If $r=0$, then $U_{<r'}=U_0$.
By Proposition \ref{bound on Delta omega} (1), we have $\Delta_{r'}\leq \Delta_{r'}^\omega=\Delta_0^\omega$.
Since $K_0(\Delta_{r'})M_{U_{>0}}(\Delta_{r'})=K_0(\Delta_{r'})M_{U_{\geq r'}}(\Delta_{r'})=K_0(\Delta_{r'})M_{U_{r'}}(\Delta_{r'})$ is locally cyclic,
we have $\Delta_{r'}\leq\Delta_0$, which proves statement $(1)$.

Now suppose that $r>0$. We claim that $\Delta_{r'}-r'\leq\Delta_r-r$.
Suppose not. Then   $\Delta_{r'}-r'\geq\Delta_r+1-r$.
Combining with Proposition \ref{bound on Delta omega} (3) we get $\Delta_r^\omega\geq\Delta_r+1$.

Let $i\in U_r$ and $j\in U_{r'}$.
By the definition of $\Delta_{r'}$ we have
$K_0(\Delta_{r'})K_j(\Delta_{r'})\subseteq K_0(\Delta_{r'})K_i(\Delta_{r'})$.
Regard $K_0(\Delta_{r'})K_i({\Delta_{r'}})$ as
a cyclic extension of $K_0({\Delta_{r'}})$.
Since $K_0(\Delta_{r'})K_j(\Delta_{r'})$ and $K_0(\Delta_{r'})K_i(\Delta_{r}+1)$ are both subfields of $K_0(\Delta_{r'})K_i(\Delta_{r'})$, for dimension reasons
$K_0(\Delta_{r'})K_i(\Delta_{r}+1)\subseteq K_0(\Delta_{r'})K_j(\Delta_{r'})$.

Since $K_0(\Delta_{r'})M_{U_{r'}}(\Delta_{r'})$ is locally cyclic, its subfield $K_0(\Delta_{r}+1)M_{U_r}(\Delta_{r}+1)$ is also locally cyclic.

As $K_0(\Delta_{r'})M_{U_{>r}}(\Delta_{r'})=K_0(\Delta_{r'})M_{U_{r'}}(\Delta_{r'})$ and $K_0(\Delta_{r'})M_{U_{r'}}(\Delta_{r'})$ is locally cyclic, its subfield
$K_0(\Delta_{r}+1)M_{U_{>r}}(\Delta_{r}+1)$ is locally cyclic.
Hence
$\Delta_r\geq\Delta_r+1$, which is a contradiction.
Therefore $\Delta_{r'}-r'\leq \Delta_r-r$.

\end{proof}


\begin{defn}\label{patchable}
Suppose that $U_r$ is nonempty.
Let $x \in \tilde{G}_\omega(K_0,K_{U_r})$. We say that $x$ is \emph{patchable} if  $x_i=0$ (mod $p^{\epsilon_0-\Delta_r}$) for all $i\in U_r$. The subgroup consisting of all patchable elements of $\tilde{G}_\omega(K_0,K_{U_r})$ is called \emph{the patchable subgroup} of $G(K_0,K_{U_r})$. We denote by $\tilde{G}(K_0,K_{U_r})$ the patchable subgroup of $G(K_0,K_{U_r})$.
\end{defn}

Note that if $U_0=\cI$, then  by above definition every element of $G^0(K_0,K')$ is \emph{patchable}.

Hence in the rest of this section we fix an $r$ such that $U_r$ is nonempty and $U_r\neq \cI$ unless we state otherwise explicitly.

In the following we show that $x$ is patchable if and only if $a_x$ defined in equation (\ref{definition of ax}) is
in $G^0(K_0,K')$.

\begin{prop}\label{projector_pachable}
Let $a_{x}$ be defined as in equation (\ref{definition of ax}).
If $a_{x}\in G^0(K_0,K')$, then $x$ is patchable.
\end{prop}
\begin{proof}
Suppose that $a_{x}=(a_1,...,a_m)\in G^0(K_0,K')$.
If $x=0$, then there is nothing to prove.
Hence in the following we assume $x\neq 0$.
By the definition of $a_x$ we have $a_1=0$ and $\cI\setminus U_r\subseteq I_{0}(a_x)$.
Set  $\epsilon_0-d=\underset{i\in U_r}{\min}\{\delta(0,a_i)\}$,
and $U_r'=\{i\in U_r|\epsilon_0-d=\delta(0,a_i)\}.$
As $x\neq 0$, we have $d>r$ and $a_x\notin D$.

By Proposition \ref{projector alg_pachable}, we have that $x$ is algebraically patchable,
i.e. $\Delta^\omega_r\geq d.$
Since $\Delta^\omega_r\geq d$, it is enough to prove that  $d$ satisfies condition (1) and (2) in Definition \ref{patching degree}.

Suppose that $U_{<r}$ is nonempty.
Let $s\in U_{<r}$ and $t\in U'_r$.
Then $e_{s,t}=e_{0,s}$.
Since $\Delta^\omega_r\geq d$, the field $K_0({d})M_{U_{r}}(d)$ is contained in $K_0({d})K_s({d})$.
By Lemma \ref{generator of bicyclic}, the field $K_0({d})M_{U_{r}}(d)$ is equal to  $K_0({d})K_t({d})$.
By  Proposition \ref{bicyclic field}, we have $K_0({d})K_t({d})$ is locally cyclic.
As $K_0({d})K_t({d})$ is locally cyclic, the field $K_0({d})M_{U_{r}}(d)$ is locally cyclic.

Now suppose that $U_{>r}$ is not empty.
For $s\in U_{\geq d}$, we have $K_s(d)=K_0(d)$.
If $U_{>r}\neq U_{\geq d}$, then there is $s\in U_{>r}\setminus U_{\geq d}$ such that
$K_0({d})M_{U_{>r}}(d)=K_0({d})K_s(d)$.
Let  $t\in U'_r$.
Then $e_{s,t}=e_{0,t}$.
Again by Proposition \ref{bicyclic field}, we have $K_0({d})K_s({d})$ is locally cyclic.
Hence $K_0({d})M_{U_{>r}}(d)$ is locally cyclic.
Therefore $\Delta_r\geq d$ and $x$ is patchable.

\end{proof}

Now we prove the converse.
\begin{prop}\label{patch}
Let $x\in\tilde{G}(K_0,K_{U_r})$ and $a_{x}$ be defined as in the equation (\ref{definition of ax}).
Then $a_x\in G^0(K_0,K')$.
\end{prop}
\begin{proof}
If $\Delta_r=r$, then clearly $a_{x}=0 \in G(K_0,K')$.
Hence we can assume $\Delta_r>r$.

As $x\in\tilde{G}(K_0,K_{U_r})$, we have $S_x=\emptyset$.
Since $S_r(\Delta_r)$ and $S_{>r}(\Delta_r)$ are also empty, by Proposition \ref{alg_patch x} we have $a_x\in G^0(K_0,K')$.
\end{proof}

\begin{lemma}\label{projection}
Let $a=(a_1,...,a_m)\in G^0_\omega(K_0,K')$ be a nonzero element.
Let $r$ be the maximum integer with the following property:
\begin{enumerate}
  \item There is some $t\in U_r\setminus I_0(a)$ such that
$\delta(a_t,0)=\underset{i\notin I_0(a)}{\min}\{\delta(a_i,0)\}$.
\end{enumerate}

Set $x=\varpi_r(a)\in\underset{i\in U_r}{\oplus}\ent/p^{\epsilon_0-r}\ent$.
Then $x\in \tilde{G}_\omega(K_0,K_{U_r})$. Moreover if $a\in G^0(K_0,K')$, then $x\in\tilde{G}(K_0,K_{U_r})$.
\end{lemma}

\begin{proof}

Clearly $x\in G_\omega(K_0,K_{U_r}).$
By the choice of $t$, we have $x\neq 0$.
Set $\epsilon_0-d=\delta(a_t,0)$.
Then $d>r$ as $x\neq 0$.
To prove that $x$ is algebraically patchable, it is enough to show that $d\leq\Delta^\omega_r$.

Set $U_r'=\{i\in U_r|\ \epsilon_0-d=\delta(0,a_i)\}.$
By the choice of $r$ we have $U'_r\neq\emptyset$.

Suppose that $U_{>r}$ is nonempty.
By the choice of $U_r$, for all $s\in U_{>r}$ we have either  $\delta(0,a_s)>\epsilon_0-d$ or $s\in I_0(a)$.
Let $s\in U_{>r}$.
Suppose $\delta(0,a_s)\leq\epsilon_0-d$.
Then $s\in I_{0}(a)$ and $e_{0,s}\geq d$.
As $e_{0,s}\geq d$, we have $K_s(d)=K_0(d)\subseteq K_0(d)K_i(d)$ for all $i$ in $U_r$.

Suppose that $\delta(0,a_s)>\epsilon_0-d$.
By Proposition \ref{bicyclic field}, we have $K_0(d)K_s(d)\subseteq K_0(d)K_i(d)$ for all $i$ in $U'_r$.
Hence by Lemma \ref{M_U equal M_U'} we have $K_0({d})M_{U_{>r}}(d)\subseteq \underset{i\in U_{r}}{\cap} K_0({d})K_i({d})$.

Suppose that $U_{<r}$ is nonempty.
Let $s\in U_{<r}$.
Then $\delta(0,a_s)\geq \epsilon_0-d$.
If $\delta(0,a_s)>\epsilon_0-d$, then by Proposition \ref{bicyclic field} we have $K_0(d)K_i(d)\subseteq K_0(d)K_s(d)$ for all $i$ in $U'_r$.

If $\delta(0,a_s)=\epsilon_0-d$, then  by Proposition \ref{bicyclic field} we have $K_0(d)K_s(d)\subseteq K_0(d)K_1(d)$. As $e_{0,s}<r$, we have $[K_0(d)K_i(d):K_0(d)]<[K_0(d)K_s(d):K_0(d)]$ for all $i\in U'_r$. Since they are both subfields of the cyclic extension $K_0(d)K_1(d)$  of $K_0(d)$, we have $K_0(d)K_i(d)\subset K_0(d)K_s(d)$ for all $i\in U'_r$.

By Lemma \ref{M_U equal M_U'}
we have $K_0({d})M_{U_{r}}(d)=K_0(d)M_{U'_r}(d)\subseteq\underset{s\in U_{<r}}{\cap} K_0({d})K_s({d})$.
Therefore $d\leq \Delta^\omega_r$ and $x$ is algebraically patchable.

Now suppose further that $a\in G^0(K_0,K')$.
Clearly $x\in G(K_0,K_{U_r}).$
Suppose that $U_{>r}$ is nonempty. Then as $K_0(d)M_{U>r}(d)$ is contained in a bicyclic extension, by Lemma \ref{generator of bicyclic} we have
$K_0(d)M_{U>r}(d)=K_0(d)K_s(d)$ for some $s\in U_{>r}$.
By the choice of $r$ and by Proposition \ref{bicyclic field} (2), we have either $K_0(d)=K_s(d)$ or $K_0(d)K_s(d)$ is locally cyclic.

Suppose that $U_{<r}$ is nonempty.
By the same argument we have $K_0(d)M_{U_r}(d)=K_0(d)K_i(d)$ for some $i\in U'_{r}$.
As $a_1=0$, by Proposition \ref{bicyclic field} (2) we have $K_0(d)K_i(d)$ is locally cyclic.
Hence $d\leq \Delta_r$.

\end{proof}

\begin{prop}\label{G^0 structure}
We have \begin{enumerate}
          \item $G_\omega(K_0,K')\simeq D \oplus\underset{r}{\oplus} \tilde{G}_\omega(K_0,K_{U_r})$,
          where $r$ runs over nonnegative integers such that $U_r$ is nonempty.
          \item $G(K_0,K')\simeq D \oplus\underset{r}{\oplus} \tilde{G}(K_0,K_{U_r})$,
          where $r$ runs over nonnegative integers such that $U_r$ is nonempty.
        \end{enumerate}
\end{prop}
\begin{proof}
To prove (1), it is sufficient to show that $G^0_\omega(K_0,K')\simeq\underset{r}{\oplus} \tilde{G}_\omega(K_0,K_{U_r})$.
By Proposition \ref{alg_patch x}, we have $G^0_\omega(K_0,K')\supseteq\underset{r}{\oplus} \tilde{G}_\omega(K_0,K_{U_r})$.

Let $a\in G^0_\omega(K_0,K')$ and $J(a)=\cI\setminus I_0(a)$.
We prove that $a\in\underset{r}{\oplus} \tilde{G}_\omega(K_0,K_{U_r})$ by induction on
$|J(a)|$.

If $|J(a)|=0$, then $a=0\in \underset{r}{\oplus} \tilde{G}_\omega(K_0,K_{U_r})$.
Suppose that the result holds when $|J(a)|<h$.

Let  $|J(a)|=h$ and let $r$ satisfy the condition (1) in Lemma \ref{projection}.
By Lemma \ref{projection}, we have $\varpi_r(a)\in \tilde{G}_\omega(K_0,K_{U_r})$.
Set $x=\varpi_r(a)$. Then $a_x\in G^0_\omega(K_0,K')$ by Proposition \ref{alg_patch x}.
Hence $(a'_i)_{i\in\cI}=a-a_x\in G^0_\omega(K_0,K')$ and $|J(a')|<h$.
By induction hypothesis $a'\in\underset{r}{\oplus} \tilde{G}_\omega(K_0,K_{U_r})$.
Hence $a\in \underset{r}{\oplus} \tilde{G}_\omega(K_0,K_{U_r})$.

Assertion (2) can be proved similarly by using Lemma \ref{projection} and  Proposition \ref{patch}.
\end{proof}

\section{The degree of freedom}
In this section, we define the \emph{algebraic degree of freedom} (resp. the \emph{degree of freedom})
 to describe the generators of the subgroups  $\tilde{G}_\omega(K_0,K_{U_r})$'s (resp. $\tilde{G}(K_0,K_{U_r})$'s).

\subsection{$l$-equivalence relations and levels}

Let $i,j\in \cI'$ and $l$ be a nonnegative integer. We say that $i,j$ are $l$-equivalent and
we write $i\underset{l}{\sim} j$ if $e_{i,j}\geq l$ or $i=j$.
As $K_i$ are cyclic, it is clear that $"\underset{l}{\sim}"$ defines an equivalence relation on any nonempty subset of $\cI'$.

For a nonempty subset $C$ of $\cI'$, denote by $n_l(C)$ be the number of $l$-equivalence classes  of $C$. In particular $n_0(C)=1$.

For each $C\subseteq \cI'$ with cardinality bigger than 1,
we define the \emph{level} of $C$ to be
the smallest integer $l$ such that $n_{l+1}(C)>1$.
For each $C=\{i\}$, we define the \emph{level} of $C$ to be $\epsilon_i$.
Denote by $L(C)$ the level of $C$.


\subsection{The degree of freedom of $U_0=\cI$}

\begin{lemma}\label{element of G}
Assume that $U_0=\cI$. Let $l=L(\cI)$ and $c$ be an equivalence class of $\cI/\underset{l+1}{\sim}$.
Let $0\leq f\leq d$ be integers satisfying the following:
\begin{itemize}
  \item (1) $M_{U_0}(d)$ is a subfield of a bicyclic extension.
  \item (2) $K_0({f})\subseteq M_{U_0}(d)$.
\end{itemize}
 For $i\in \cI$, set $x=(x_1,...,x_m)$ as follows:
\begin{equation}
x_j=\left\{
\begin{array}{ll}
p^{\epsilon_0-f}, & \hbox{if $j\in c$}\\
0, & \hbox{otherwise.}
\end{array}
\right .
\end{equation}
Then $x\in G_\omega(K_0,K')$.
\end{lemma}
\begin{proof}
By the definition of $l$,
the field $M_{U_0}(d)$ is cyclic if and only if $d\leq l$.
In this case we have $f=0$ and $x=0\in G_\omega(K_0,K')$.
Hence we assume $M_{U_0}(d)$ is bicyclic in the following.

As there are more than one equivalence classes in $\cI/\underset{l+1}{\sim}$, the set $\cI\setminus c$ is non-empty
and $I_0(x)=\cI\setminus c$.
Since $L(\cI)=l$, we have $e_{s,t}\geq l$ for any $s$, $t\in\cI$. Moreover for $s\in c$ and $t\in \cI\setminus c$ we have $e_{s,t}=l$.
By Lemma \ref{generator of bicyclic} we have $M_{\cI}(d)=K_s(d)K_t(d)$ for any $s\in c$ and $t\in \cI\setminus c$.

Let $s$, $t$ be as above.
Let $v\in\Omega$ be a place where $M_{\cI}(d)^v$ is  locally cyclic at $v$.
We claim that either $K_0(f)\otimes_k K_s(d)^v$ is a product of copies of $K_s(d)^v$ or $K_0(f)\otimes_k K_t(d)^v$ is a product of copies of $K_t(d)^v$.

Let $\gamma$ (resp. $\gamma_i$'s) be the integer such that $K_0(f)^v$ (resp. $K_i(d)^v$) is a product of  extensions of degree $p^\gamma$ (resp. $p^{\gamma_i}$) of $k^v$.
Suppose that $K_0(f)^v\otimes_{k^v} K_s(d)^v$ is not a product of copies of $K_s(d)^v$.
Then since $M_\cI(d)^v$ is a product of cyclic extensions,
we have $\gamma>\gamma_s$.
If $\gamma>\gamma_t$, then $M_{\cI}(d)^v=(K_s(d)K_t(d))^v$ is a product of fields of degree less than $\gamma$, which is a contradiction as $K_0(f)\subseteq M_{\cI}(d)$.
Hence $\gamma<\gamma_t$ and $K_0(f)^v\otimes_{k^v} K_t(d)^v$ is a product of copies of $K_t(d)^v$ by the cyclicity of $M_{\cI}(d)^v$.

Since $M_{\cI}(d)^v$ is  locally cyclic at almost all $v\in\Omega_k$, we have $x\in G_\omega(K_0,K')$.
\end{proof}

\medskip

\begin{remark}\label{deg of freedom of U_0}
If $M_{\cI}(d)$ in Lemma \ref{element of G} is locally cyclic, then by the above proof we have $x\in G(K_0,K')$.
\end{remark}

\medskip

Keep the notation defined as above.
The element $x$ has order $p^f$.
Suppose that $f>0$.
Then $d>L(U_0)$ by the above proof.
As $K_0(f)K_i(d)\subseteq M_{U_0}(d)$, for dimension reasons  $f\leq d-L(U_0)$.
On the other hand, by Lemma \ref{subfield of bicyclic} if $K_0(f)\subseteq M_{U_0}(d)$, then $K_0(f)\subseteq M_{U_0}(f+L(U_0))$. Hence we can choose $d=f+L(U_0)$ and define \emph{the algebraic degree of freedom $f^\omega_{U_0}$ of $U_0$} to be the largest $f$ such that $f$ and $d=f+L(U_0)$ satisfy the conditions in Lemma \ref{element of G}. By Proposition \ref{bicyclic field} \emph{the algebraic degree of freedom} $f^\omega_{U_0}$ is the maximal possible order of a class function on $U_0/\underset{l+1}{\sim}$ which lies in ${G}_\omega(K_0,K')$.

\subsection{General cases}

Inspired by the definition of $f^\omega_{U_0}$, for $U_r$ nonempty we define the algebraic degree of freedom of $U_r$ to describe the generators of $\tilde{G}_\omega(K_0,K_{U_r})$.
Briefly speaking, the group $\tilde{G}_\omega(K_0,K_{U_r})$ is generated by class functions on $U_r/\underset{l}{\sim}$ for
 $l>L(U_r)$. The order of such a generator is called \emph{the degree of freedom}.

\begin{defn} \label{alg deg of freedom of U_r}
For a nonempty $U_r$, let $l_r=L(U_r)$.
Let $f\leq \Delta^\omega_r$ be a nonnegative integer satisfying the following:
\begin{enumerate}
  \item  The field $M_{U_r}(f+l_r-r)$ is a subfield of a bicyclic extension.
  \item  $K_0({f})\subseteq M_{U_r}(f+l_r-r)$.
\end{enumerate}
Then we set $f^\omega_{U_r}$ to be the largest $f\leq\Delta^\omega_r$ satisfying above conditions.
We call $f^\omega_{U_r}$ \emph{the algebraic degree of freedom} of $U_r$.
\end{defn}

\begin{remark}
Note that $f=r$ always satisfies the conditions in Definition \ref{alg deg of freedom of U_r}.
Hence we have $\Delta^\omega_r\geq f^\omega_{U_r}\geq r$.
\end{remark}

For $h\geq L(U_r)$ and a class $c_0$ of $U_r/\underset{h}{\sim}$,
 we define by recursion \emph{the algebraic degree of freedom of $c\in c_0/\underset{h+1}{\sim}$} as follows.
\begin{defn}
Keep the notation defined as above.
Let $f\leq f^\omega_{c_0}$ be a nonnegative integer satisfying the following:
\begin{enumerate}
  \item  The field $M_{c}(f+L(c)-r)$ is a subfield of a bicyclic extension.
  \item  $K_0({f})\subseteq M_{c}(f+L(c)-r)$.
\end{enumerate}
Then we set $f^\omega_{c}$ to be the largest $f\leq f^\omega_{c_0}$ satisfying above conditions.
We call $f^\omega_{c}$ \emph{the algebraic degree of freedom} of $c$.
\end{defn}

Inspired by Remark \ref{deg of freedom of U_0} we define similarly \emph{the degree of freedom}.

\begin{defn}
For a nonempty $U_r$, let $h\geq L(U_r)$ and $c\in U_r/\underset{h}{\sim}$.
We define \emph{the degree of freedom} $f_{c}$ of $c$ to be the maximum integer $f\leq f^\omega_{c}$ such that $M_{c}(f+L(c)-r)$ is locally cyclic.

\end{defn}

\begin{remark}\label{unramified extension}
From the above definition, we see that $f_{c}=f^\omega_{c}$ if
$M_{c}(f^\omega_{c}+L(c)-r)$ is locally cyclic (e.g. unramified) over $k$.
\end{remark}

\begin{prop}\label{expression of M as K_0K_i}
For a nonempty $U_r$, let  $h\geq L(U_r)$ and $c\in U_r/\underset{h}{\sim}$.
For all integers $r\leq f\leq f^\omega_{c}$ and $i\in c$, we have $M_{c}(f+L(c)-r)=K_0(f)K_i(f+L(c)-r)$.
\end{prop}
\begin{proof}
For $f=r$, it is trivial. Hence we assume $f>r$ in the following. This means that
$|c|>1$ and $M_{c}(f+L(c)-r)$ is not cyclic.
By the definition of $f^\omega_{c}$ and Lemma \ref{generator of bicyclic}, there are $i,j\in c$ such that
$e_{i,j}=L(c)$ and
$M_{c}(f^\omega_{c}+L(c)-r)=K_i(f^\omega_{c}+L(c)-r)K_j(f^\omega_{c}+L(c)-r)$.
As $K_0(f)\subseteq M_{c}(f^\omega_{c}+L(c)-r)$, by Lemma \ref{subfield of bicyclic}
we have $K_0(f)\subseteq K_i(f+L(c)-r)K_j(f+L(c)-r)$.
By dimension reasons $K_0(f)K_i(f+L(c)-r)=K_i(f+L(c)-r)K_j(f+L(c)-r)$.

Since $M_{c}(f+L(c)-r)$ is contained in $K_i(f^\omega_{c}+L(c)-r)K_j(f^\omega_{c}+L(c)-r)$,
by Lemma \ref{generator of bicyclic} we have $M_{c}(f+L(c)-r)=K_i(f+L(c)-r)K_j(f+L(c)-r)$.
Hence $M_{c}(f+L(c)-r)=K_0(f)K_i(f+L(c)-r)$.
\end{proof}

In the following we assume that $U_r$ is nonempty.
We prove a generalization of Lemma \ref{element of G}.

\begin{prop}\label{lemma of set of generator of G}
Let $l\geq L(U_r)$ and $c\in U_r/\underset{l}{\sim}$. Let  $r\leq f\leq f^\omega_{c}$ be an integer.
Set $l_1=L(c)$, $c_1\in c/\underset{l_1+1}{\sim}$, and \[\cS_{c}=\{v\in\Omega_k|M_{c}(f+L(c)-r)^v\mbox{ is not locally cyclic at $v$}\}.\]
Then $\Omega_k\setminus \cS_{c}\subseteq(\underset{j\in c_1}{\cap}\Sigma_j^{\epsilon_0-f})\cup (\underset{j\in U_r\setminus c_1}{\cap}\Sigma_j^{\epsilon_0-f}$).
\end{prop}

\begin{proof}

If $f=r$, then clearly $v\in \Sigma^{\epsilon_0-r}_j$ for all $j\in U_r$ and all $v\in\Omega_k$.
Hence in the following we assume that  $f>r$, which implies that $|U_r|>1.$
Note that for any equivalence class $c_0\in U_r/\underset{h}{\sim}$ with $L(U_r)\leq h\leq l$ and $c_0\supseteq c$,
we have $f\leq f^\omega_{c}\leq f^\omega_{c_0}$.

We prove the statement by induction on $l$.
Consider the case where $l=L(U_r)$.
By definition $c=U_r$.
As there are more than one equivalence class in $U_r/\underset{l+1}{\sim}$, the set $U_r\setminus c_1$ is non-empty.  By Lemma \ref{generator of bicyclic} we have $M_{U_r}(f+l-r)=K_s({f+l-r})K_t({f+l-r})$ for any $s\in c_1$ and $t\in U_r\setminus c_1$. By Proposition \ref{expression of M as K_0K_i} we have
$K_0({f})K_s({f+l-r})=M_{U_r}({f+l-r})
=K_0({f})K_t({f+l-r})$.

By Lemma \ref{locally cyclic condition} for all $v\in \Omega_k\setminus \cS_{U_r}$,
either $K_0(f)\otimes_k K_s(f+l-r)^v$ is a product of copies of $K_s(f+l-r)^v$ or $K_0(f)\otimes_k K_t(f+l-r)^v$ is a product of copies of $K_t(f+l-r)^v$. Hence $\Omega_k\setminus \cS_{U_r}\subseteq(\underset{j\in c_1}{\cap}\Sigma_j^{\epsilon_0-f})\cup(\underset{j\in U_r\setminus c_1}{\cap}\Sigma_j^{\epsilon_0-f})$, and the statement is true for $l=L(U_r)$.

Suppose that the statement is true for  $l=h>L(U_r)$.
Let $l=h+1$.
If $c$ is also an equivalence class of $U_r/\underset{h}{\sim}$, then the statement is true by the induction hypothesis.

Now suppose that $c\notin U_r/\underset{h}{\sim}$.
Let $v\in\Omega_k\setminus \cS_{c}$. It suffices to show that if $v\notin\underset{j\in c_1}{\cap}\Sigma_j^{\epsilon_0-f}$, then $v\in\underset{j\in U_r\setminus c_1}{\cap}\Sigma_j^{\epsilon_0-f}$.
Suppose that $v\notin(\underset{j\in c_1}{\cap}\Sigma_j^{\epsilon_0-f})$.
We first prove that $v\in (\underset{j\in c\setminus c_1}{\cap}\Sigma_j^{\epsilon_0-f})$.

Let $s\in c_1$ such that $v\notin \Sigma_s^{\epsilon_0-f}$.
By Lemma \ref{generator of bicyclic} the field $M_{c}(f+l_1-r)$ is equal to $K_{s}({f+l_1-r})K_j({f+l_1-r})$ for any $j\in c\setminus c_1$.
Hence $M_{c}(f+l_1-r)=K_0({f})K_j(f+l_1-r)=K_0({f})K_s({f+l_1-r})$.
Since $M_{c}(f+l_1-r)^v$ is a product of cyclic extensions of $k_v$ and $v\notin \Sigma_s^{\epsilon_0-f}$,
by Lemma \ref{locally cyclic condition} we have
$K_0({f})\otimes_k K_j({f+l_1-r})^v$ is a product of copies of $K_j({f+l_1-r})^v$.
This implies $v\in (\underset{j\in c\setminus c_1}{\cap}\Sigma_j^{\epsilon_0-f})$.

Next we show that $v\in(\underset{j\in U_r\setminus c}{\cap}\Sigma_j^{\epsilon_0-f})$.
As $c\notin U_r/\underset{h}{\sim}$, there is some $c_o\in U_r/\underset{h}{\sim}$ such that $c\varsubsetneq c_0$.
This implies that $L(c_0)=h$ and $c\in c_0/\underset{h+1}{\sim}$.
By induction hypothesis  $\Omega_k\setminus \cS_{c_0}\subseteq (\underset{j\in U_r\setminus c}{\cap}\Sigma_j^{\epsilon_0-f})\cup(\underset{j\in c}{\cap}\Sigma_j^{\epsilon_0-f})$.

As $h<l_1$, by Proposition \ref{expression of M as K_0K_i} we have $M_{c_0}(f+h-r)\subseteq M_{c}(f+l_1-r)$.
Hence
$\Omega_k\setminus\cS_{c}\subseteq \Omega_k\setminus\cS_{c_0}$,
which implies that $v\in\Omega_k\setminus \cS_{c_0}$.

Since  $c_1\subseteq c$ and   $v\notin(\underset{j\in c_1}{\cap}\Sigma_j^{\epsilon_0-f})$, $v\notin(\underset{j\in c}{\cap}\Sigma_j^{\epsilon_0-f})$.
Hence $v\in\underset{j\in U_r\setminus c}{\cap}\Sigma_j^{\epsilon_0-f}$ by induction hypothesis.
Combining this with the fact that $v\in (\underset{j\in c\setminus c_1}{\cap}\Sigma_j^{\epsilon_0-f})$, we have $v\in (\underset{j\in U_r\setminus c_1}{\cap}\Sigma_j^{\epsilon_0-f})$.

\end{proof}

\begin{coro}\label{set of generator of  G omega}
Let $l\geq L(U_r)$ and $c\in U_r/\underset{l}{\sim}$.
Set $l_1=L(c)$, $c_1\in c/\underset{l_1+1}{\sim}$ and
 $x^\omega_{c_1}=(x_i)_{i\in U_r}\in \underset{i\in U_r}{\oplus}\ent/p^{\epsilon_0-r}\ent$ as follows:
\begin{equation}
x_j=\left\{
\begin{array}{ll}
p^{\epsilon_0-f^\omega_{c}}, & \hbox{for all $j\in c_1,$}\\
0, & \hbox{otherwise.}
\end{array}
\right .
\end{equation}
Then $x^\omega_{c_1}\in G_\omega(K_0,K_{U_r})$.
\end{coro}
\begin{proof}
It is a direct consequence of Proposition \ref{lemma of set of generator of G}.
\end{proof}

\begin{coro}\label{set of generator of  G}
Keep the notation as in Cor.\ref{set of generator of  G omega}.
Set $x_{c_1}=(x_i)_{i\in U_r}\in \underset{i\in U_r}{\oplus}\ent/p^{\epsilon_0-r}\ent$ as follows:
\begin{equation}
x_j=\left\{
\begin{array}{ll}
p^{\epsilon_0-f_{c}}, & \hbox{for all $j\in c_1,$}\\
0, & \hbox{otherwise.}
\end{array}
\right .
\end{equation}
Then $x_{c_1}\in G(K_0,K_{U_r})$.
\end{coro}
\begin{proof}
It is a direct consequence of Proposition \ref{lemma of set of generator of G}.
\end{proof}

\section{The computation of $\sha^2_\omega(k,\hat{T}_{L/k})$ and $\sha^2(k,\hat{T}_{L/k})$}
In this section we use the (algebraic) patching degrees and
the (algebraic) degrees of freedom to describe the groups
$\sha^2(k,\hat{T}_{L/k})$ and $\sha_\omega^2(k,\hat{T}_{L/k})$.

\subsection{Generators of algebraic patchable subgroups and patchable subgroups}
For $U_r$ nonempty, set $$x^\omega_{U_r}=(p^{\epsilon_0-\Delta^\omega_r})_{i\in U_r}\in G_\omega(K_0,K_{U_r});$$ and
$$x_{U_r}=(p^{\epsilon_0-\Delta_r})_{i\in U_r}\in G(K_0,K_{U_r}).$$
In the following we show that  the elements $x_{c_1}$'s (resp. $x^\omega_{c_1}$) defined
in Corollary \ref{set of generator of G} (resp. Corollary \ref{set of generator of  G omega}) are generators of $\tilde{G}(K_0,K_{U_r})$ (resp. $\tilde{G}_\omega(K_0,K_{U_r})$).

\begin{prop}\label{a system of generator of patchable}
For a nonempty $U_r$, we have the following:
\begin{enumerate}
  \item The algebraic patchable subgroup $\tilde{G}_\omega(K_0,K_{U_r})$ is generated by $x^\omega_{c}$ for all $l\geq L(U_r)$ and $c\in U_r/\underset{l}{\sim}$.
  \item The patchable subgroup $\tilde{G}(K_0,K_{U_r})$ is generated by $x_{c}$ for all
  $l\geq L(U_r)$ and $c\in U_r/\underset{l}{\sim}$.
\end{enumerate}
\end{prop}

\begin{proof}
Let $x=(x_i)_{i\in U_r}\in \tilde{G}_\omega(K_0,K_{U_r})\subseteq\underset{i\in U_r}{\oplus}(\ent/p^{\epsilon_0-r}\ent)$.
Let $t$ be the smallest index in $U_r$.
After modifying $x$ by a multiple of $x_{U_r}$, we can assume $x_t=0$.

Let $I(x)=(I_0(x),...,I_{p^{\epsilon_0-r}-1}(x))$ be the partition of $U_r$ associated to $x$.
Set $J=U_r\setminus I_0(x)$.
We prove the proposition by induction on $|J|$.
If $|J|=0$, then it is clear that $x=0\in \langle x^\omega_{c}\rangle$.
Let $h$ be a positive integer,
and  suppose that the statement is true for all $|J|<h$.

For $|J|=h$,
let $\epsilon_0-d=\underset{i\in J}{\min}\{\delta(0,x_i)\}$.
As $x$ is patchable, we have $d\leq \Delta^\omega_r$.
Let $J'=\{i\in J|\ \delta(0,x_i)=\epsilon_0-d\}$.
Let $l$ be the smallest integer such that there is $c\in U_r/\underset{l}{\sim}$ contained in $J'$.
Pick $i\in c$.
We claim that for
all $r\leq l_0< l$ and $c_0\in U_r/\underset{l_0}{\sim}$ containing $c$, the field $M_{c_0}(d+{{L}}(c_0)-r)$ is a subfield of a bicyclic extension.

By the choice of $l_0$, we have
$c_0\nsubseteq J'$.
Set $J_0=J'\cap c_0$.
Pick $j\in J_0$ and $i\in c_0\setminus J_0$ such that $$e_{i,j}=\max\{e_{s,t}|\ s\in J_0,  t\in c_0\setminus J_0\}.$$
By Lemma \ref{bicyclic field} we  have $F_{d,i,j}=K_0(d)K_i({d+e_{i,j}-r})= K_0(d)K_{j}({d+e_{i,j}-r})$.
Again by Lemma \ref{bicyclic field}
for any $s\in c_0\setminus J_0$, we  have $F_{d,s,j}=K_0(d)K_j({d+e_{s,j}-r})\subseteq F_{d,i,j}$.
Similarly $F_{d,i,s}\subseteq F_{d,i,j}$ for all $s\in J_0$.

Note that by definition $L(c_0)=\min\{e_{t,t'}|\ t,t'\in c_0\}.$
Since $L(c_0)\leq e_{s,j}\leq e_{i,j}$, we have  $K_s({d+L(c_0)-r})K_j({d+L(c_0)-r})\subseteq F_{d,s,j}$ for all
$s\in c_0\setminus J_0$.

By a similar argument, we have $K_s({d+L(c_0)-r})K_{i}({d+L(c_0)-r})\subseteq F_{d,i,s}$ for all $s\in J_0$.
Hence $M_{c_0}(d+L(c_0)-r)\subseteq F_{d,i,j}$.

Next we show that $d\leq f^\omega_{c_0}$.
As $M_{c_0}(d+L(c_0)-r)$ is a subfield of a bicyclic extension, by Lemma \ref{generator of bicyclic}
there are $s$ and $t$ such that
$M_{c_0}(d+L(c_0)-r)=K_{s}(d+L(c_0)-r)K_{t}(d+L(c_0)-r)$.
Moreover we can choose $s$, $t\in c_0$ such that $s\notin J_0$ and $t\in J_0$.

To see this, first suppose that
$J_0$ is contained in some $c'\in c_0/\underset{L(c_0)+1}{\sim}$. Then we can
pick $s\in c_0\setminus c'$ and pick $t\in J_0$.

Suppose that $J_0\nsubseteq c'$ for any $c'\in c_0/\underset{L(c_0)+1}{\sim}$.
Then pick $s\in c_0\setminus J_0$.
Let  $c'$ be the class of  $c_0/\underset{L(c_0)+1}{\sim}$ containing $s$.
Since $J_0$ is not contained in $c'$,  $J_0\setminus c'$ is nonempty.
Pick $t\in J_0\setminus c'$.
Then $e_{s,t}=L(c_0)$. By Lemma \ref{generator of bicyclic}, we have $M_{c_0}(d+L(c_0)-r)=F_{d,s,t}$.

By Lemma \ref{bicyclic field} $K_0(d)\subseteq M_{c_0}(d+L(c_0)-r)$,
so we have $d\leq f^\omega_{c_0}$.
In particular for $c_0\in U_r/\underset{l-1}{\sim}$ containing $c$, we have $d\leq f^\omega_{c_0}$.
By Corollary \ref{set of generator of  G omega}, there is an integer $n$ such that the $i$-th coordinate of $nx^\omega_{c}$ is $x_i$.
Since $c\subseteq J'$, the number of non-zero coordinates of $x-nx^\omega_{c}$ decreases by at least one.
By the induction hypothesis, the element $x-nx^\omega_{c}$ is generated by patchable diagonal elements and $x^\omega_{c'}$ for $l'\geq L(U_r)$ and $c'\in U_r/\underset{l'}{\sim}$.
Statement (1) then follows.

Suppose further that $x\in\tilde{G}(K_0,K_{U_r})$.
Then by Lemma \ref{bicyclic field} $M_{c_0}(d+L(c_0)-r)=F_{d,s,t}$ is locally cyclic.
Hence $d\leq f_{c_0}$.
By similar argument we get statement (2).
\end{proof}

\begin{theo}\label{description of patchable group}
Suppose that $U_r$ is nonempty. Then
\begin{enumerate}
  \item $\tilde{G}_\omega(K_0,K_{U_r})\simeq \ent/p^{\Delta^\omega_r-r}\ent\oplus \underset{l\geq L(U_r)}{\oplus} \underset{c\in U_r/\underset{l}{\sim} }{\oplus}(\ent/p^{f^\omega_{c}-r}\ent)^{n_{l+1}(c)-1}.$
  \item $\tilde{G}(K_0,K_{U_r})\simeq \ent/p^{\Delta_r-r}\ent\oplus \underset{l\geq L(U_r)}{\oplus} \underset{c\in U_r/\underset{l}{\sim} }{\oplus}(\ent/p^{f_{c}-r}\ent)^{n_{l+1}(c)-1}.$
\end{enumerate}

\end{theo}

\begin{proof}

By Proposition \ref{a system of generator of patchable}, the group $\tilde{G}_\omega(K_0,K_{U_r})$ is generated by the $x^\omega_{c}$ for $l\geq L(U_r)$ and $c\in U_r/\underset{l}{\sim}$.

It is clear that the cyclic group $\langle x^\omega_{U_r}\rangle\simeq \ent/ p^{\Delta^\omega_r-r}\ent$.
For $l\geq L(U_r)$, $c\in U_r/\underset{l}{\sim}$, and $c_1\in c/\underset{L(c)+1}{\sim}$,
the group $\langle x^\omega_{c_1}\rangle$
is isomorphic to $\ent/ p^{f^\omega_{c}-r}\ent$.

For $U_r$ we have $$\underset{c\in U_r/\underset{L(U_r)+1}{\sim}}{\sum} x^\omega_{c}=p^{{\Delta^\omega_r}-f^\omega_{U_r}}x^\omega_{U_r}.$$
Let $c_0\in U_r/\underset{l_0}{\sim}$ and $c\in c_0/\underset{L(c_0)+1}{\sim}$.
Set $l=L(c_0)+1$. If $n_{l+1}(c)>1$, then we have the relation
$$\underset{c_1\in c/\underset{l+1}{\sim}}{\sum} x^\omega_{c_1}=p^{f^\omega_{c_0}-f^\omega_{c}}x^\omega_{c},$$
We choose $n_{l+1}(c)$-1 distinct classes in $c/\underset{l+1}{\sim}$. Let $c_i$ for $1\leq i<n_{l+1}(c)$ be these classes. Then $x^\omega_{c_i}$'s
generate a group isomorphic to $(\ent/p^{f^\omega_{c}-r}\ent)^{n_{l+1}(c)-1}$ and this group is disjoint from the group generated by $ x^\omega_{c'}$ for $ L(U_r)\leq h\leq l$ and $c'\in U_r/\underset{h}{\sim}$ and
by $ x^\omega_{c'}$ for $c'\in U_r/\underset{l+1}{\sim}$ for $c'\nsubseteq c$.

Hence $\tilde{G}(K_0,K_{U_r})\simeq \ent/p^{\Delta^\omega_r-r}\ent\underset{l\geq L(U_r)}{\oplus} \underset{c\in U_r/\underset{l}{\sim} }{\oplus}(\ent/p^{f^\omega_{c}-r}\ent)^{n_{l+1}(c)-1}$.

One can prove (2) by a similar argument.
\end{proof}

For $U_0=\cI$, we get the group structure of $\sha^1(k,T_{L/k})$ immediately from the above theorem.
\begin{coro}\label{sha for U_0=I}
Suppose that $U_0=\mathcal{I}$. Then
\begin{enumerate}
  \item $\sha^2_\omega(k,\hat{T}_{L/k})\simeq\underset{l\geq L(\mathcal{I})}{\oplus} \underset{c\in \mathcal{I}/\underset{l}{\sim} }{\oplus}(\ent/p^{f^\omega_{c}}\ent)^{n_{l+1}(c)-1}$.
  \item $\sha^2(k,\hat{T}_{L/k})\simeq\underset{l\geq L(\mathcal{I})}{\oplus} \underset{c\in \mathcal{I}/\underset{l}{\sim} }{\oplus}(\ent/p^{f_{c}}\ent)^{n_{l+1}(c)-1}$.
\end{enumerate}
\end{coro}

\begin{proof}
The arguments for (1) and (2) are similar.
We show (2) here.

As $U_0=\cI$, we have $\Delta_0=\epsilon_0$ and $G(K_0,K')=\tilde{G}(K_0,K_{U_0})$.
By Theorem \ref{description of patchable group} the group $G(K_0,K')\simeq \ent/p^{\epsilon_0}\ent \underset{l\geq L(\mathcal{I})}{\oplus} \underset{c\in \mathcal{I}/\underset{l}{\sim} }{\oplus}(\ent/p^{f_{c}}\ent)^{n_{l+1}(c)-1}.$
As the diagonal group $D$ is isomorphic to $\ent/p^{\epsilon_0}\ent$, we have $$\sha^2(k,\hat{T}_{L/k})\simeq\underset{l\geq L(\mathcal{I})}{\oplus} \underset{c\in \mathcal{I}/\underset{l}{\sim} }{\oplus}(\ent/p^{f_{c}}\ent)^{n_{l+1}(c)-1}.$$

\end{proof}

\subsection{The Tate-Shafarevich groups}
For $i\in U_r$ and $l\geq L(U_r)$, set $a^\omega_{c}=(a_j)_{j\in\mathcal{I}}$ to be the embedding of $x^\omega_{c}=(x_j)_{j\in U_r}$ in $G_\omega(K_0,K')$ as follows:
\begin{equation}
a_j=\left\{
\begin{array}{ll}
x_j, & \hbox{for all $j\in U_r,$}\\
0, & \hbox{otherwise.}
\end{array}
\right .
\end{equation}

We define $a_{c}=(a_j)_{j\in\mathcal{I}}$ to be the embedding of $x_{c}=(x_j)_{j\in U_r}$ in $G(K_0,K')$ in the same way.

\begin{prop}\label{a system of generator of G}
We have the following:
\begin{enumerate}
  \item The group $G_\omega(K_0,K')$ is generated by the diagonal group $D$ and the $a^\omega_{c}$'s defined as above.
  \item The group $G(K_0,K')$ is generated by the diagonal group $D$ and the $a_{c}$'s defined as above.
\end{enumerate}
\end{prop}


\begin{proof}
Let $a=(a_i)_{i\in \mathcal{I}}\in G_\omega(K_0,K')$.
After modifying by a diagonal element, we can assume that $a_1=0$.
By Proposition \ref{G^0 structure} we have $a\in\underset{r}{\oplus}\tilde{G}_\omega(K_0,K_{U_r})$.
Then $a$ is generated by $D$ and $x^\omega_{c}$'s by Proposition \ref{a system of generator of patchable}.



A similar argument proves (2).

\end{proof}



\begin{theo}\label{group structure of G omega}
Keep the notations as above. Then we have
 $$G_\omega(K_0,K')\simeq\ent/p^{\epsilon_0}\ent \underset{r\in\mathcal{R}\setminus\{0\}}{\oplus}\ent/p^{\Delta^\omega_r-r}\ent\underset{r\in\mathcal{R}}{\oplus} \underset{l\geq L(U_r)}{\oplus} \underset{c\in U_r/\underset{l}{\sim} }{\oplus}(\ent/p^{f^\omega_{c}-r}\ent)^{n_{l+1}(c)-1};$$
 and
 $$G(K_0,K')\simeq\ent/p^{\epsilon_0}\ent \underset{r\in\mathcal{R}\setminus\{0\}}{\oplus}\ent/p^{\Delta_r-r}\ent \underset{r\in\mathcal{R}}{\oplus} \underset{l\geq L(U_r)}{\oplus} \underset{c\in U_r/\underset{l}{\sim} }{\oplus}(\ent/p^{f_{c}-r}\ent)^{n_{l+1}(c)-1}.$$
As a consequence, we have
 $$\sha^2_\omega(k,\hat{T}_{L/K})\simeq\underset{r\in\mathcal{R}\setminus\{0\}}{\oplus}\ent/p^{\Delta^\omega_r-r}\ent \underset{r\in\mathcal{R}}{\oplus} \underset{l\geq L(U_r)}{\oplus} \underset{c\in U_r/\underset{l}{\sim} }{\oplus}(\ent/p^{f^\omega_{c}-r}\ent)^{n_{l+1}(c)-1};$$
 and

$$\sha^1(k,T_{L/K})\simeq\underset{r\in\mathcal{R}\backslash\{0\}}{\oplus}\ent/p^{\Delta_r-r}\ent \underset{r\in\mathcal{R}}{\oplus} \underset{l\geq L(U_r)}{\oplus} \underset{c\in U_r/\underset{l}{\sim} }{\oplus}(\ent/p^{f_{c}-r}\ent)^{n_{l+1}(c)-1}.$$
\end{theo}

\begin{proof}
By Proposition \ref{a system of generator of G}, the group $G_\omega(K_0,K')$ is generated by the diagonal group $D$ and the group $\underset{r\in\mathcal{R}}{\oplus}\tilde{G}_\omega(K_0,K_{U_r})$.
If $U_0=\mathcal{I}$, then it is  Theorem \ref{description of patchable group}.

Suppose that $U_0\neq\mathcal{I}$.
Set $a_{\mathcal{I}}=(1,...,1)$, which is a generator of $D$.
Then we have the relation
$$\underset{r\in\mathcal{R} }{\sum} p^{\Delta^\omega_r-\Delta^\omega_0}a^\omega_{U_r}=p^{\epsilon_0-\Delta^\omega_0}a_{\mathcal{I}}.$$
Note that by Proposition \ref{bound on Delta omega} (1) and (2), we have $\Delta^\omega_0>0$ and $\Delta^\omega_r-\Delta^\omega_0\geq 0$.
Hence $p^{\epsilon_0-\Delta^\omega_0}a_{\mathcal{I}}$ and
$a^\omega_{U_0}$ is nonzero.
It is clear that the element $p^{\epsilon_0-\Delta^\omega_0}a_{\mathcal{I}}$ generates the intersection
$D\cap \underset{r\in\mathcal{R}}{\oplus}\tilde{G}_\omega(K_0,K_{U_r})$.
Hence $$G_\omega(K_0,K')\simeq D\underset{l\geq L(U_0)}{\oplus} \underset{c\in U_0/\underset{l}{\sim} }{\oplus}(\ent/p^{f^\omega_{c}}\ent)^{n_{l+1}(c)-1}\underset{r\in\mathcal{R}\backslash\{0\}}{\oplus}\tilde{G}_\omega(K_0,K_{U_r}).$$

Applying Proposition \ref{bound on Delta} (1) and (2) instead of Proposition \ref{bound on Delta omega}, one proves in a similar way the statement of $G(K_0,K')$.
\end{proof}

\begin{remark}
Let $\mathcal{K}$ be a minimal Galois extension of $k$ which splits $T_{L/k}$
and denote its Galois group by $\mathcal{G}$.
An alternative way to calculate $\sha^2_\omega(\mathcal{G},\hat{T}_{L/k})$ is to
express the degree of freedom and patching degree in terms of the group structure of $\mathcal{G}$.
Then one can use the method in \cite{BP19} to get $\sha^2_\omega(\mathcal{G},\hat{T}_{L/k})$ from
$\sha^2(l,M)$ for some finite extension $l$ and some $\Gal(k_s/k)$-module $M$.
\end{remark}

\section{Examples}
In this section, we give some examples where more explicit descriptions of the groups $\sha^2(k,\hat{T}_{L/k})$ and $\sha^2_\omega(k,\hat{T}_{L/k})$ are obtained.
We first note the following case.
\begin{prop}\label{criterion for sha trivial}
If $\underset{i\in U_0}{\cap} K_0K_i=K_0$, then $\sha^2(k,\hat{T}_{L/k})=\sha_\omega^2(k,\hat{T}_{L/k})=0$.
\end{prop}

\begin{proof}
It is enough to show that $\sha_\omega^2(k,\hat{T}_{L/k})=0$.
Let $l=L(U_0)$.
If $f^\omega_{U_0}\neq 0$,
then  $K_0({f^\omega_{U_0}})M_{U_0}(f^\omega_{U_0}+l)$ is bicyclic and by Proposition \ref{expression of M as K_0K_i} it is contained in
$\underset{i\in U_0}{\cap} K_0K_i$,
which is a contradiction. Therefore $f^\omega_{U_0}=0$.
If $U_0=\mathcal{I}$, then $\sha^2_\omega(k,\hat{T}_{L/k})=0$ by Corollary \ref{sha for U_0=I}.

Suppose that $U_0\neq \mathcal{I}$.
Choose $r>0$ such that $U_r$ is nonempty.
Since $\underset{i\in U_0}{\cap} K_0K_i=K_0$, we have $\Delta^\omega_r=r$.
Hence $f^\omega_c=r$ for all $c\in U_r/\underset{l}{\sim}$.
By Theorem \ref{group structure of G omega}  $\sha^2_\omega(k,\hat{T}_{L/k})=0$.
\end{proof}

\begin{example}Let $k=\rat$ and $\zeta_n$ be a primitive $n$-th root of unity.
Let $p_0$,...,$p_m$ be distinct odd primes and $n_i$ be positive integers.
Set $K_i=\rat(\zeta_{{p_i}^{n_i}})$ for $0\leq i\leq m$.
Then $K_i$ are cyclic extensions. Since $\underset{i\in\cI}\cap K_0K_i=K_0$, by Proposition \ref{criterion for sha trivial} the group $\sha^2_\omega(k,\hat{T}_{L/k})=0$.
\end{example}

\begin{prop}\label{linearly disjoint}
Suppose that $K_i$ are linearly disjoint extensions of $k$ for all $i\in \cI'$.
Let $f$ be the maximum integer such that $M_{\cI'}(f)$ is a subfield of a bicyclic extension;
and $f'$ be the maximum integer such that $M_{\cI'}(f')$ is bicyclic and locally cyclic.
Then $\sha^2_\omega(k,\hat{T}_{L/k})\simeq(\ent/p^f\ent)^{m-1}$ and $\sha^2(k,\hat{T}_{L/k})\simeq(\ent/p^{f'}\ent)^{m-1}$.
\end{prop}
\begin{proof}
Since $K_i$ are disjoint extensions for all $i\in \cI'$, we have $U_0=\cI$, $L(U_0)=0$
and $n_{1}(U_0)=m$.
Then by definition we have $f^\omega_{U_0}=f$ and $f_{U_0}=f'$.
The proposition follows from Corollary \ref{sha for U_0=I}.
\end{proof}

\begin{example}\label{2-power-bicyclic}
Let $k=\rat(i)$.
Let $K_0=k(\sqrt[4]{17})$, $K_1=k(\sqrt[4]{17\times 13})$ and $K_2=k(\sqrt[4]{13})$.
Then $M_{\cI'}(2)$ is a bicyclic extension of $k$ with Galois group $\ent/4\ent\times\ent/4\ent$.
Hence $\sha^2_\omega(k,\hat{T}_{L/k})\simeq\ent/4\ent$.

It is clear that $M_{\cI'}(2)_v$  is a product of cyclic extensions if $v$ is an unramified place.
Let $\mathcal{P}$ be the prime ideal associated to $v$.
If $M_{\cI'}(2)$ is ramified at $v$, then $\mathcal{P}\cap\ent\in\{(2),(13),(17)\}$.
Since $17$ is not a $4$-th power root in $\rat_{13}$, we have $M_{\cI'}(2)_{17}$ is not cyclic.

It is easy to check that $M_{\cI'}(1)$  is locally cyclic.
Hence by Proposition \ref{linearly disjoint} we have $\sha^2(k,\hat{T}_{L/k})\simeq\ent/2\ent$.
\end{example}

\begin{example}\label{2-power-bicyclic}
Let $k=\rat(i)$.
Let $K_0=k(\sqrt[4]{17})$, $K_1=k(\sqrt[4]{17\times 409})$ and $K_2=k(\sqrt[4]{409})$.
Then $M_{\cI'}(2)$ is a bicyclic extension of $k$ with Galois group $\ent/4\ent\times\ent/4\ent$.
Hence $\sha^2_\omega(k,\hat{T}_{L/k})\simeq\ent/4\ent$.

We claim that $M_{\cI'}(2)$  is locally cyclic.
Let $v\in\Omega_k$.
It is clear that $M_{\cI'}(2)_v$  is a product of cyclic extensions if $v$ is an unramified place.
Let $\mathcal{P}$ be the prime ideal associated to $v$.
If $M_{\cI'}(2)$ is ramified at $v$, then $\mathcal{P}\cap\ent\in\{(2),(17),(409)\}$.
However $409$ and $17$ are quartic residues of each other, and $17$ has a $4$-th root in $\rat_2$.
Therefore  $M_{\cI'}(2)$ is locally cyclic and $f_{U_0}=2$.
By Proposition \ref{linearly disjoint} we have $\sha^2(k,\hat{T}_{L/k})\simeq\ent/4\ent$.
In this case the weak approximation holds for $T_{L/k}$-torsors with a $k$-point.
\end{example}

\begin{prop}\label{sha omega of subfields of bicyclic}
Let $F$ be a bicyclic extension of $k$ with Galois group $\ent/p^n\ent\times\ent/p^n\ent$.
Let $K_i$ be distinct cyclic subfields of $F$ with degree $p^n$.
Then $$\sha^2_\omega(k,\hat{T}_{L/k})\simeq \underset{r\in\mathcal{R}\setminus\{0\}}{\oplus}\ent/p^{n-r}\ent \underset{r\in\mathcal{R}}{\oplus} \underset{l\geq L(U_r)}{\oplus} \underset{c\in U_r/\underset{l}{\sim} }{\oplus}(\ent/p^{n-l}\ent)^{n_{l+1}(c)-1}.$$
\end{prop}
\begin{proof}
Regard $F$ as a cyclic extension of $K_0$.

For a nonempty $U_r$ and all $i\in U_r$, the field $K_0K_i$ is the unique degree $p^{n-r}$ extension of $K_0$ contained in $F$.
Hence $\Delta_r^\omega=n$ for all $r\in\cR$.

For $l\geq L(U_r)$, the field $M_{c}(n)$ is contained in $F$ and its Galois group is isomorphic to $\ent/p^n\ent\times \ent/p^{n-L(c)}$.
We claim that $f^\omega_{c}=n-L(c)+r$.
Regard $F$ as a cyclic field extension of $K_i$.
As subfields of $F$, both $K_0(n-L(c)+r)K_i$ and $M_{c}(n)$ are cyclic extensions of $K_i$ of degree $p^{n-L(c)}$. Hence $K_0(n-L(c)+r)K_i=M_{c}(n)$ and $f^\omega_{c}=n-L(c)+r$.

For a class $c\in U_r/\underset{l}{\sim}$, we have $n_{l+1}(c)>1$ if and only if $L(c)=l$.
The proposition then follows.

\end{proof}

\begin{example}
Let $k=\rat(i)$.
Let $K_0=k(\sqrt[4]{13})$, $K_1=k(\sqrt[4]{17})$, $K_2=k(\sqrt[4]{13\times 17^2})$.
Then $1\in U_0$ and $2\in U_1$.
By Proposition \ref{sha omega of subfields of bicyclic}, we have $\sha^2_\omega(k,\hat{T}_{L/k})\simeq \ent/2\ent$.
As the field $K_0K_2$ is locally cyclic, we have
$\Delta_1=2$.
Hence $\sha^1(k,T_{L/k})\simeq \ent/2\ent$.
In this case the weak approximation holds for $T_{L/k}$-torsors with a $k$-point.
\end{example}

\bigskip
Ting-Yu Lee

Technische Universit\"{a}t Dortmund

Fakult\"{a}t f\"{u}r Mathematik

Lehrstuhl LSVI

Vogelpothsweg 87

44227 Dortmund, Germany

\medskip

tingyu.lee@gmail.com

\end{document}